\documentclass[10pt]{article}
\usepackage{graphicx}
\usepackage{mathrsfs}
\usepackage{amsfonts}
\usepackage{amssymb}
\usepackage{amsmath}
\usepackage{amsthm}
\usepackage[bookmarks=true]{hyperref}

\DeclareMathOperator{\dist}{dist}

\newtheorem{theorem}{Theorem}[section]
\newtheorem{lemma}[theorem]{Lemma}
\newtheorem{conjecture}[theorem]{Conjecture}
\newtheorem{proposition}[theorem]{Proposition}
\newtheorem{corollary}[theorem]{Corollary}
\theoremstyle{remark}
\newtheorem*{remark}{Remark}

\title{Multi-arm incipient infinite clusters in 2D: scaling limits and winding
numbers}
\author{Chang-Long Yao \thanks{Academy of Mathematics and Systems Science,
CAS, Beijing, China (E-mail: deducemath@126.com)}}

\begin{document}
\maketitle
\markboth{Sample paper for the {\protect\ntt\lowercase{amsmath}} package}
{Sample paper for the {\protect\ntt\lowercase{amsmath}} package}
\renewcommand{\sectionmark}[1]{}

\begin{abstract}
We study the alternating $k$-arm incipient infinite cluster (IIC) of
site percolation on the triangular lattice $\mathbb{T}$.  Using
Camia and Newman's result that the scaling limit of critical site
percolation on $\mathbb{T}$ is CLE$_6$, we prove the existence of
the scaling limit of the $k$-arm IIC for $k=1,2,4$.  Conditioned on
the event that there are open and closed arms connecting the origin
to $\partial \mathbb{D}_R$, we show that the winding number variance
of the arms is $(3/2+o(1))\log R$ as $R\rightarrow \infty$, which
confirms a prediction of Wieland and Wilson (2003).  Our proof uses
two-sided radial SLE$_6$ and coupling argument. Using this result we
get an explicit form for the CLT of the winding numbers, and get
analogous result for the 2-arm IIC, thus improving our earlier
result.


\end{abstract}

\section{Introduction}
Percolation is a central model of probability theory and statistical
physics, see \cite{9,23} for background and \cite{10} for a summary
of recent progress.  For bond percolation on $\mathbb{Z}^d$, there
is almost surely no infinite open cluster at the critical point when
$d=2$ or $d>10$ (see the recent work \cite{100}), and is conjectured
that this is the case whenever $d\geq 2$. The term ``incipient
infinite cluster'' (IIC) has been used by physicists to refer to the
large-scale connected clusters present in critical percolation, and
was defined mathematically by Kesten \cite{16} in two dimensions.
Roughly speaking, IIC is obtained by conditioning on the event that
there is an open path connecting the origin to the boundary of the
box with radius $n$ centered at the origin, and letting
$n\rightarrow\infty$. Following Kesten's spirit, Damron and
Sapozhnikov introduced multi-arm IIC in \cite{4}.  We will give the
definitions of these IICs later.

In fact, IIC is a very natural and robust object that can be
constructed in many different ways.  We introduce some natural
constructions for dimension two as follows.  In \cite{16}, Kesten
gave an alternative way to construct the IIC: Take $p>p_c$,
condition on the cluster of the origin to be infinite, and let
$p\rightarrow p_c$. J\'{a}rai \cite{15} showed that if we choose a
site uniformly from the largest cluster or the spanning clusters in
$[-n,n]^2$, and let $n\rightarrow\infty$, then we get the IIC.  In
\cite{27} J\'{a}rai also proved that the invasion percolation
cluster looks asymptotically like the IIC, when viewed from an
invaded site $v$, in the limit $|v|\rightarrow\infty$.  Similarly,
Damron and Sapozhnikov \cite{4} showed that the invasion percolation
cluster looks asymptotically like the 2-arm IIC (resp. 4-arm IIC),
when viewed from a site $v$ belonging to the backbone (resp.
outlets), in the limit $|v|\rightarrow\infty$.  Recently, Hammond,
Pete and Schramm \cite{11} defined a local time measure on the
exceptional set of dynamical percolation, and showed that at a
typical time with respect to this measure, the percolation
configuration has the law of IIC.  For IIC in high dimensions, see
\cite{13,14}, where it was also shown that several related and
natural constructions lead to the same object.

In this paper, we will study the scaling limit of IIC for site
percolation on the triangular lattice $\mathbb{T}$ and the winding
numbers of the arms.  Before giving our main results, we wish to
introduce some related works in the literature.

The scaling limit of IIC has been extensively studied in recent
years, and it has turned out to be useful in understanding the
discrete model.  We list a few related works in the following:

\begin{itemize}
\item \emph{Percolation in high dimensions.}  Van der Hofstad conjectured in \cite{13} that
the scaling limit of IIC above 6 dimensions is infinite canonical
super-Brownian motion (ICSBM), which corresponds to the canonical
measure of super-Brownian motion conditioned on non-extinction.
ICSBM consists of a single infinite Brownian motion path together
with super-Brownian motions branching off from this path.  In
\cite{30},  it is showed that the scaling limit of the backbone of
the high-dimensional IIC is Brownian motion.  The scaling limit of
another version of high-dimensional IIC is conjectured to be
integrated super-Brownian excursion (ISE) by Hara and Slade
\cite{12}.  Using the lace expansion, they obtained strong evidence
for their conjecture in \cite{12}.
\item \emph{Oriented percolation in high dimensions.}  The existence of
the IIC for sufficiently spread-out oriented percolation on
$\mathbb{Z}^d\times \mathbb{Z}_+$ above $4+1$ dimensions has been
proved by van der Hofstad, den Hollander, and Slade \cite{31}.  Van
der Hofstad \cite{13} proved that ICSBM is the scaling limit of the
IIC.
\item \emph{Percolation on a regular tree.} The IIC on a regular tree
was constructed by Kesten in \cite{29}.  It has a simple structure,
and can be viewed as an infinite backbone from the origin with
critical percolation clusters attached to it.  Very recently, Angel,
Goodman and Merle \cite{1} proved that the scaling limit of the IIC
(w.r.t. the pointed Gromov-Hausdorff topology) is a random
$\mathbb{R}$-tree with a single end.
\end{itemize}

Motivated by a question from Beffara and Nolin \cite{2}, in
\cite{24} we proved a CLT for the winding numbers of alternating
arms crossing the annulus $A(l,n)$ (as $n\rightarrow\infty$ and $l$
fixed) for critical percolation on $\mathbb{T}$ and $\mathbb{Z}^2$.
Using this, we also got a CLT for corresponding multi-arm IIC in
\cite{24}.  However, the exact estimate for the winding number
variance was not given in that paper.  Based on numerical
simulations,  Wieland and Wilson \cite{50} made a conjecture on the
winding number variance of Fortuin-Kasteleyn contours (and more
generally, the winding at points where $k$ paths come together),
including the above case.  The conjecture seems hard, to our
knowledge, it has been verified rigorously on only a few particular
cases.  For example, conditioned on the event that there are 2
(resp. 3) disjoint loop-erased random walks starting at the
neighbors of the origin and ending at the unit circle centered at
the origin in $\eta\mathbb{Z}^2$, Kenyon \cite{70} (see also
``Remarks on LERW" in \cite{50}) showed that the winding number
variance of the paths is $(1/2+o(1))\log(1/\eta)$ (resp.
$(2/9+o(1))\log(1/\eta)$) as $\eta\rightarrow 0$.  The interested
reader is referred to the Introduction of \cite{24} for a more
general discussion and references on winding numbers.

The rest of the paper is organized as follows. Section \ref{s11}
introduces the basic notation used throughout the paper, and gives
the definitions of $k$-arm IIC measure and arm events for CLE$_6$.
Section \ref{s12} gives our main results, together with the main
ideas in their proofs.  In Section \ref{s21}, we define the uniform
metric, which is related to the convergence in distribution. Section
\ref{s22} collects different versions of coupling arguments that
will be used.  Section \ref{s23} gives basic properties of arm
events, including a generalized quasi-multiplicativity.  Section
\ref{s3} provides proofs of scaling-limit results for multi-IIC.  In
Section \ref{s41}, we introduce two-sided radial SLE and give second
moment estimate for its winding number.  We study convergence of
discrete exploration to SLE$_6$ in Section \ref{s42}, moment bounds
on the winding of discrete exploration in Section \ref{s43}, and
decorrelation of winding in Section \ref{s44}, which will enable us
to translate the winding number result for two-sided radial SLE$_6$
to percolation.  Section \ref{s45} provides proofs of the winding
number results for the arms.

\subsection{The model and notation}\label{s11}

Let $\mathbb{T}=(\mathbb{V},\mathbb{E})$ denote the triangular
lattice, where $\mathbb{V}:=\{x+ye^{\pi
i/3}\in\mathbb{C}:x,y\in\mathbb{Z}\}$ is the set of sites, and
$\mathbb{E}$ is the set of bonds, connecting adjacent sites.
Throughout the paper, we will focus on critical site percolation on
$\eta\mathbb{T}$ with small mesh size $\eta>0$, where each site is
chosen to be blue (open) or yellow (closed) with probability $1/2$,
independently of each other.  Let $P=P^{\eta}$ denote the
corresponding product probability measure on the set of
configurations.  We also represent the measure as a (blue or yellow)
random coloring of the faces of the dual hexagonal lattice
$\eta\mathbb{H}$, and view the sites of $\eta\mathbb{T}$ as the
hexagons of $\eta\mathbb{H}$.  Further, let $H_v$ denote the regular
hexagon centered at $v\in\mathbb{V}(\mathbb{T})$ with side length
$1/\sqrt{3}$ with two of its sides parallel to the imaginary axis.

A \emph{path} is a sequence $v_0,\ldots,v_n$ of distinct sites of
$\mathbb{T}$ such that $v_{i-1}$ and $v_i$ are neighbors for all
$i=1,\ldots,n$.  A \emph{boundary path} (or \emph{b-path}) is a
sequence $e_0,\ldots,e_n$ of distinct edges of $\mathbb{H}$
belonging to the boundary of a cluster and such that $e_{i-1}$ and
$e_i$ meet at a vertex of $\mathbb{H}$ for all $i=1,\ldots,n$.  A
\emph{circuit} is a path whose first and last sites are neighbors.
For a circuit $\mathcal {C}$, define
\begin{equation*}
\overline{\mathcal {C}}:= \mathcal {C}\cup\mbox{ interior sites of
}\mathcal {C}.
\end{equation*}

A \emph{color sequence} $\sigma$ is a sequence
$(\sigma_1,\sigma_2,\dots,\sigma_k)$ of ``blue" and ``yellow" of
length $k$.  We use the letters ¡°B¡± and ¡°Y¡± to encode the
colors.  We identify two sequences if they are the same up to a
cyclic permutation.

We say that a finite set $D$ of hexagons is \emph{simply connected}
if both $D$ and its complement are connected.  For a simply
connected set $D$ of hexagons, we denote by $\Delta D$ its
\emph{external site boundary}, or \emph{s-boundary} (i.e., the set
of hexagons that do not belong to $D$ but are adjacent to hexagons
in $D$), and by $\partial D$ the topological boundary of $D$ when
$D$ is considered as a domain of $\mathbb{C}$.  We will call a
bounded, simply connected subset $D$ of $\mathbb{T}$ a \emph{Jordan
set} if $\Delta D$ is a circuit.

Given a Jordan set $D\subset\mathbb{T}$, for any vertex
$v\in\mathbb{H}$ that belongs to $\partial D$, if the edge incident
on $v$ that is not in $D$ does not belong to a hexagon in $D$, we
call $v$ an \emph{e-vertex}.

Given a Jordan set $D$ and two e-vertices $a,b$ in $\partial D$, we
denote by $\partial_{a,b}D$ the portion of $\partial D$ traversed
counterclockwise from $a$ to $b$, and call it the \emph{right
boundary}; the remaining part of the boundary is denoted by
$\partial_{b,a}D$ and is called the \emph{left boundary}.
Analogously, the portion of $\Delta_{a,b}D$ of $\Delta D$ whose
hexagons are adjacent to $\partial_{a,b}D$ is called the \emph{right
s-boundary} and the remaining part the \emph{left s-boundary}.
Imagine coloring blue all the hexagons in $\Delta_{a,b}D$ and yellow
all those in $\Delta_{b,a}D$. Then, for any percolation
configuration inside $D$, there is a unique b-path $\gamma$ from $a$
to $b$ which separates the blue cluster adjacent to $\Delta_{a,b}D$
from the yellow cluster adjacent to $\Delta_{b,a}D$. We call
$\gamma=\gamma_{D,a,b}$ a \emph{percolation exploration path}.

Given a Jordan domain $D$ of the plane, we denote by $D^{\eta}$ the
largest Jordan set of hexagons of $\eta\mathbb{H}$ that is contained
in $D$.  For two distinct points $a,b\in\partial D$, we let
$\gamma_{D,a,b}^{\eta}:=\gamma_{D^{\eta},a_{\eta},b_{\eta}}$, where
$a_{\eta}$ (resp. $b_{\eta}$) is the e-vertex in $\partial D^{\eta}$
closest to $a$ (resp. $b$).  If there are two such vertices closest
to $a$ (resp. $b$), we choose the first one encountered going
clockwise (resp. counterclockwise) along $\partial D^{\eta}$.
Further, let
$\partial_{a,b}D^{\eta}:=\partial_{a_{\eta},b_{\eta}}D^{\eta}$ and
$\Delta_{a,b}D^{\eta}:=\Delta_{a_{\eta},b_{\eta}}D^{\eta}$.

For a domain $D$, let $\overline{D}:=D\cup\partial D$.  For a
topological annulus $A=\overline{D}_2\backslash D_1$ ($D_1$ and
$D_2$ are Jordan domains) whose boundary is composed of two simple
loops in the plane, we denote by $\partial_1A$ (resp. $\partial_2A$)
the \emph{inner} (resp. \emph{outer}) \emph{boundary} of $A$, and
let $A^{\eta}:=\overline{D_2^{\eta}}\backslash D_1^{\eta}$.

Define the \emph{disc} and \emph{annulus} as follows: for
$0<r<R,z\in\mathbb{C}$,
\begin{align*}
&\mathbb{D}_R(z):=\{x\in
\mathbb{C}:|x-z|<R\},~~\mathbb{D}_R:=\mathbb{D}_R(0),~~\mathbb{D}:=\mathbb{D}_1;\\
&A(z;r,R):=\overline{\mathbb{D}_R(z)}\backslash\mathbb{D}_r(z),~~A(r,R):=A(0;r,R).
\end{align*}

Now let us define the arm events for percolation.  For a topological
annulus $A$ whose boundary is composed of two simple loops, denote
by $\mathcal {A}_{\sigma}^{\eta}(A)=\mathcal
{A}_{k,\sigma}^{\eta}(A)$ the event that there exist $|\sigma|=k$
disjoint monochromatic paths (\emph{arms}) in $A^{\eta}$ connecting
the two boundary pieces of $A^{\eta}$, whose colors are those
prescribed by $\sigma$, when taken in counterclockwise order.  For
$|\sigma|\leq 6$, given a Jordan domain $D$ with a point $z\in D$,
let $\mathcal {A}_{\sigma}^{\eta}(z;D)$ denote the event that there
exist $|\sigma|$ disjoint arms connecting $\partial D^{\eta}$ and
the hexagon in $\eta\mathbb{H}$ whose center is closest to $z$ (if
there are more than one such hexagons, we choose a unique one by
some deterministic method), whose colors are those prescribed by
$\sigma$, when taken in counterclockwise order. For any $\eta\leq
r<R$ and $z\in \mathbb{C}$, write
$$\mathcal {A}_{\sigma}^{\eta}(z;r,R):=\mathcal
{A}_{\sigma}^{\eta}(A(z;r,R)).$$  For short, let $\mathcal
{A}_{\sigma}^{\eta}(r,R)=\mathcal {A}_{\sigma}^{\eta}(0;r,R)$ and
let $\mathcal {A}_1^{\eta}=\mathcal {A}_B^{\eta}$, $\mathcal
{A}_2^{\eta}=\mathcal {A}_{BY}^{\eta}$, $\mathcal
{A}_4^{\eta}=\mathcal {A}_{BYBY}^{\eta}$.

The IIC was defined by Kesten \cite{16} as follows.  It is shown in
\cite{16} that the limit
$$\nu_1^{\eta}(E):=\lim_{R\rightarrow\infty}P^{\eta}(E|\mathcal {A}_1^{\eta}(\eta,R))$$
exists for any event $E$ that depends on the state of finitely many
sites in $\eta\mathbb{T}$.  The unique extension of $\nu_1^{\eta}$
to a probability measure on configurations of $\eta\mathbb{T}$
exists and we call $\nu_1^{\eta}$ the \emph{IIC measure} or
\emph{1-arm IIC measure}.  Then, Damron and Sapozhnikov introduced
multi-arm IIC measures in \cite{4}.  Let $k=2,4$.  For every
cylinder event $E$, it is shown in Theorem 1.6 in \cite{4} the limit
\begin{equation*}
\nu_k^{\eta}(E):=\lim_{R\rightarrow\infty}P^{\eta}(E|\mathcal
{A}_k^{\eta}(\eta,R))
\end{equation*}
exists.  The unique extension of $\nu_k^{\eta}$ to a probability
measure on the configurations of $\eta\mathbb{T}$ exists.  We call
$\nu_k^{\eta}$ the \emph{$k$-arm IIC measure}.  A curve
$\gamma[0,1]$ is called a \emph{loop} if $\gamma(0)=\gamma(1)$.  All
percolation interfaces under $\nu_k^{\eta}$ induce a probability
measure on the loops in the one-point compactification
$\hat{\mathbb{C}}$ of $\mathbb{C}$, denoted by $\mu_k^{\eta}$.  We
postpone precise definitions of the space of loops and the topology
of weak convergence till Section \ref{s21}.  We also call
$\mu_k^{\eta}$ the $k$-arm IIC measure.

Given a percolation configuration, we assign a direction to each
edge of $\eta\mathbb{H}$ belonging to the boundary of a cluster in
such a way that the hexagon to the right of the edge with respect to
the direction is blue.  To each b-path $\gamma$, we can associate a
direction according to the direction of the edges in the path.
Denote by $\Gamma_B(\gamma)$ (resp., $\Gamma_Y(\gamma)$) the set of
blue (resp., yellow) hexagons adjacent to $\gamma$; we also let
$\Gamma(\gamma):=\Gamma_B(\gamma)\cup\Gamma_Y(\gamma)$.

For any Jordan domain $D$, let $P_D^{\eta}$ denote the percolation
law in $D^{\eta}$ with \emph{monochromatic (blue) boundary
condition}, that is, all the sites in $\Delta D^{\eta}$ are blue.
Then the percolation interfaces under $P_D^{\eta}$ induce a law on
the loops in $\overline{D}$, denoted by $\mu_D^{\eta}$.

In Camia and Newman \cite{32}, the following theorem is shown:

\begin{theorem}[\cite{32}]\label{t2}
Let $D$ be a Jordan domain. As $\eta\rightarrow 0$, $\mu_D^{\eta}$
converges in law, under the topology induced by metric (\ref{e2}),
to a probability distribution $\mu_D$ on collections of continuous
nonsimple loops in $\overline{D}$.
\end{theorem}

The continuum nonsimple loop process in Theorem \ref{t2} is just the
\emph{full scaling limit} introduced by Camia and Newman
\cite{3,32}.  Since it is also called the \emph{conformal loop
ensemble} CLE$_6$ in \cite{21} (for the general CLE$_\kappa$,
$8/3\leq \kappa\leq 8$, see \cite{21,22}), we just call it CLE$_6$
(in $\overline{D}$) in the present paper.

For simplicity, let $P_R^{\eta}:=P_{\mathbb{D}_R}^{\eta}$,
$\mu_R^{\eta}:=\mu_{\mathbb{D}_R}^{\eta}$ and
$\mu_R:=\mu_{\mathbb{D}_R}$.

We need to define arm events for CLE$_6$ in a way that makes them
measurable and equal to the limit of the probability of
corresponding arm events for percolation as $\eta\rightarrow 0$. Now
we express the arm events $\mathcal {A}_k^{\eta}(r,R),k=1,2,4$ for
$\mu_R^{\eta}$ in terms of loops (cluster interfaces).  See Figure
\ref{fig3}.

\begin{figure}
\begin{center}
\includegraphics[height=0.22\textwidth]{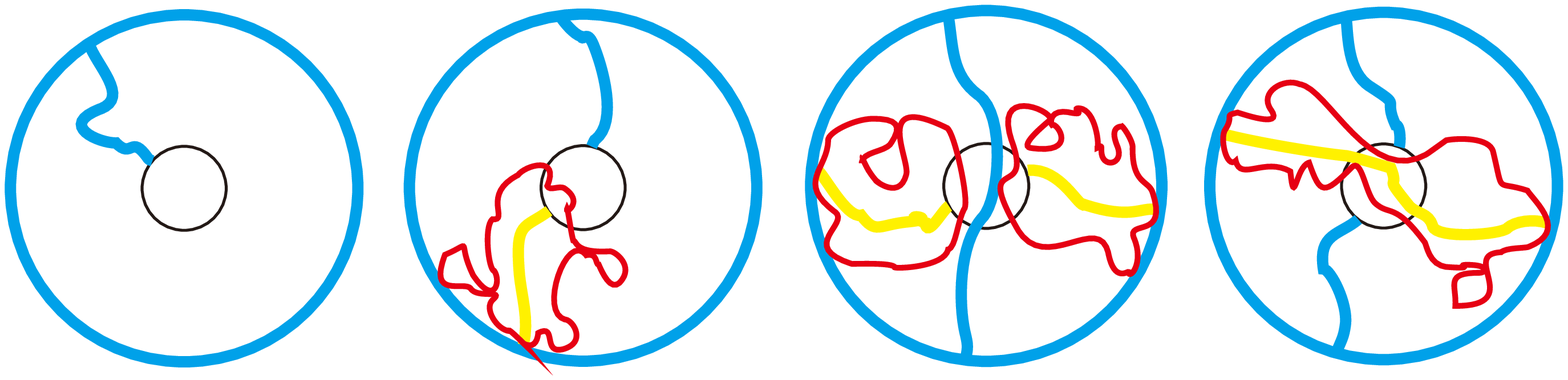}
\caption{Illustration of arm events with monochromatic blue boundary
condition.  The red loops are the outer boundaries of the clusters
containing yellow arms.  The first panel indicates $\mathcal
{A}_1^{\eta}(r,R)$. The second panel indicates $\mathcal
{A}_2^{\eta}(r,R)$. The last two panels indicate $\mathcal
{A}_4^{\eta}(r,R)$.}\label{fig3}
\end{center}
\end{figure}

\begin{itemize}
\item  It is well-known that the complement of $\mathcal
{A}_1^{\eta}(r,R)$ is that there exists a yellow circuit surrounding
the origin in $A^{\eta}(r,R)$.  Since $\mu_R^{\eta}$ has
monochromatic blue boundary condition, the outer boundary of the
cluster containing this yellow circuit is in
$A^{\eta}(r,R)\backslash
\partial_1 A^{\eta}(r,R)$, and has
counterclockwise direction.  So, we have
\begin{equation}\label{e11}
\mathcal {A}_1^{\eta}(r,R)=\left\{
\begin{aligned}
&\mbox{There exists no counterclockwise
loop surrounding}\\
&\mbox{the origin in $A^{\eta}(r,R)\backslash
\partial_1 A^{\eta}(r,R)$}
\end{aligned}
\right\}.
\end{equation}
In fact, a simple observation leads to that
\begin{equation*}
\mathcal {A}_1^{\eta}(r,R)=\left\{
\begin{aligned}
&\mbox{There exits neither
counterclockwise loop nor clockwise}\\
&\mbox{loop surrounding the origin in $A^{\eta}(r,R)\backslash
\partial_1 A^{\eta}(r,R)$}
\end{aligned}
\right\}.
\end{equation*}
\item Assume that $\mathcal
{A}_2^{\eta}(r,R)$ holds, then there exist a blue arm and a yellow
arm connecting $\partial_1A^{\eta}(r,R)$ and
$\partial_2A^{\eta}(r,R)$. The outer boundary of the cluster
containing the yellow arm must intersect with both of the two
boundary pieces of $A^{\eta}(r,R)$. Conversely, if there exists a
counterclockwise loop $\gamma$ in $A^{\eta}(r,R)$ intersecting both
of the two boundary pieces of $A^{\eta}(r,R)$, we can find a blue
arm in $\Gamma_B(\gamma)$ and a yellow arm in $\Gamma_Y(\gamma)$,
which connect the two boundary pieces of $A^{\eta}(r,R)$. Hence,
\begin{equation}\label{e12}
\mathcal {A}_2^{\eta}(r,R)=\left\{
\begin{aligned}
&\mbox{There exists a counterclockwise
loop in $\overline{\mathbb{D}_R^{\eta}}$, which}\\
&\mbox{intersects with both $\partial_1A^{\eta}(r,R)$ and
$\partial_2A^{\eta}(r,R)$}
\end{aligned}
\right\}.
\end{equation}
\item Denote by $\mathcal {A}_4^{\eta,B}(r,R)$ (resp.
$\mathcal {A}_4^{\eta,Y}(r,R)$) the event that there are four
alternating arms in $A^{\eta}(r,R)$ connecting
$\partial_1A^{\eta}(r,R)$ and $\partial_2A^{\eta}(r,R)$, and the two
blue (resp. yellow) arms are in the same cluster in
$\overline{\mathbb{D}_R^{\eta}}$.  It is clear that $\mathcal
{A}_4^{\eta}(r,R)=\mathcal {A}_4^{\eta,B}(r,R)\cup\mathcal
{A}_4^{\eta,Y}(r,R)$.  If $\mathcal {A}_4^{\eta,B}(r,R)$ occurs,
there exist two counterclockwise loops in
$\overline{\mathbb{D}_R^{\eta}}$, which intersect with both
$\partial_1A^{\eta}(r,R)$ and $\partial_2A^{\eta}(r,R)$; if
$\mathcal {A}_4^{\eta,Y}(r,R)$ occurs, there exists a
counterclockwise loop in $\overline{\mathbb{D}_R^{\eta}}$, which is
composed of two curves $\gamma_1$ and $\gamma_2$: $\gamma_1$ starts
at $a\in\partial_2A^{\eta}(r,R)$ and ends at
$b\in\partial_2A^{\eta}(r,R)$, $\gamma_2$ starts at $b$ and ends at
$a$, both $\gamma_1$ and $\gamma_2$ intersect with
$\partial_1A^{\eta}(r,R)$. In fact, it is easy to see that
\begin{equation}\label{e13}
\mathcal {A}_4^{\eta}(r,R)=\left\{
\begin{aligned}
&\mbox{There exist two counterclockwise
loops in $\overline{\mathbb{D}_R^{\eta}}$, which}\\
&\mbox{intersect with both $\partial_1A^{\eta}(r,R)$ and
$\partial_2A^{\eta}(r,R)$;
or there}\\
&\mbox{exists a counterclockwise loop in
$\overline{\mathbb{D}_R^{\eta}}$, which
is composed}\\
&\mbox{of two curves $\gamma_1$ and $\gamma_2$: $\gamma_1$ starts at $a\in\partial_2A^{\eta}(r,R)$ and}\\
&\mbox{ends at $b\in\partial_2A^{\eta}(r,R)$, $\gamma_2$ starts at
$b$ and ends at $a$, both}\\
&\mbox{$\gamma_1$ and $\gamma_2$ intersect with
$\partial_1A^{\eta}(r,R)$}
\end{aligned}
\right\}.
\end{equation}
\end{itemize}

This leads us to define arm events $\mathcal {A}_k(r,R),k=1,2,4$ for
$\mu_R$ as follows:

\begin{align*}
&\mathcal {A}_1(r,R):=\{\mbox{There exists no counterclockwise loop
surrounding the origin in $A(r,R)$}\}.\\
&\mathcal {A}_2(r,R):=\left\{
\begin{aligned}
&\mbox{There exists a counterclockwise loop in
$\overline{\mathbb{D}}_R$, which}\\
&\mbox{intersects with both $\partial_1A(r,R)$ and
$\partial_2A(r,R)$}
\end{aligned}
\right\}.\\
&\mathcal {A}_4(r,R):=\left\{
\begin{aligned}
&\mbox{There exist two counterclockwise loops
in $\overline{\mathbb{D}}_R$, which intersect}\\
&\mbox{with both $\partial_1A(r,R)$ and $\partial_2A(r,R)$; or there
exists a counterclockwise}\\
&\mbox{loop in $\overline{\mathbb{D}}_R$, which is composed of two
curves $\gamma_1$ and $\gamma_2$:
$\gamma_1$ starts}\\
&\mbox{at $a\in\partial_2A(r,R)$ and ends at $b\in\partial_2A(r,R)$,
$\gamma_2$ starts at $b$ and ends}\\
&\mbox{at $a$, both $\gamma_1$ and $\gamma_2$ intersect with
$\partial_1A(r,R)$}
\end{aligned}
\right\}.
\end{align*}

Given two Jordan domains $D$ and $D'$ with $\overline{D'}\subset D$,
similarly to the definitions of $\mathcal {A}_k(r,R)$ for $\mu_R$,
one can define arm events $\mathcal {A}_k(\overline{D}\backslash
D')$ for $\mu_D$.

In this paper, we sometimes omit the superscript $\eta$ of
$P^{\eta}$ and $\gamma^{\eta}$ when it is clear that we are talking
about the the discrete percolation model.  $C,C_1,C_2,\ldots$ and
$\alpha,\beta$ denote positive finite constants that may change from
line to line or page to page according to the context.

\subsection{Main results}\label{s12}
Our main results include two parts, the first part is about the
existence and conformal invariance of the $k$-arm IIC scaling limit,
the second part is about the variance estimate and CLT for the
winding numbers of the arms, conditioned on the 2-arm event and
under the 2-arm IIC measure, respectively.

\subsubsection{Scaling limit of $k$-arm IIC}

\begin{theorem}\label{t6}
Let $k=1,2,4$.  Let $D$ be a Jordan domain with a point $z\in D$.
Let $\{D_n\}$ be a sequence of Jordan domains such that $z\in
D_n,\overline{D}_n\subset D$ and the diameter of $D_n$ converges to
zero as $n\rightarrow \infty$.
\begin{itemize}
\item As $\eta\rightarrow 0$ and
$n\rightarrow\infty$, $\mu_D^\eta[\cdot|\mathcal {A}_k^{\eta}(z;D)]$
and $\mu_D[\cdot|\mathcal {A}_k(\overline{D}\backslash D_n)]$
converge in law, under the topology induced by metric (\ref{e2}), to
the same probability measure, denoted by $\mu_{k,D,z}$.
\item Furthermore, let $D'$ be a Jordan domain and let
$f:\overline{D}\rightarrow \overline{D'}$ a continuous function that
maps $D$ conformally onto $D'$. Let $z':=f(z)$. Then the image of
$\mu_{k,D,z}$ under $f$ has the same law as $\mu_{k,D',z'}$.
\end{itemize}
\end{theorem}

We call $\mu_{k,D,z}$ the \emph{scaling limit of $k$-arm IIC pinned
at $z$ in $D$}, which can be considered as a conditioned version of
CLE$_6$.  In \cite{34}, the authors constructed \emph{CLE$_{\kappa}$
in $D$ conditioned on the event that $z$ is in the gasket} (i.e.,
the set of points that are not surrounded by any loop in
CLE$_{\kappa}$) for $8/3<\kappa\leq 4$.  One can view $\mu_{1,D,z}$
as CLE$_6$ in $D$ conditioned on the event that $z$ is in the
gasket.  We write $\mu_{k,R}:=\mu_{k,\mathbb{D}_R,0}$.

\begin{remark}
Using Theorem \ref{t6} and Theorem 4 in \cite{32}, it is not hard to
show that $\mu_{k,D,z}$ inherits some domain Markov property from
CLE$_6$.  It is expected that analogs of Propositions 4.3 and 4.4 in
\cite{34} for domain Markov property of simple CLE in the punctured
disc also hold for $\mu_{k,D,z}$.
\end{remark}

For a domain $D$, we denote by $I_D$ the mapping (on $\Omega$ or
$\Omega_R$, see the definitions in Section \ref{s21}) in which all
portions of curves that exit $\overline{D}$ are removed.  Let
$\hat{I}_D$ be the same mapping lifted to the space of probability
measures on $\Omega$ or $\Omega_R$.

\begin{theorem}\label{t4}
There exists a unique probability measure $\mu_k$ on the space
$\Omega$ of collections of continuous curves in $\hat{\mathbb{C}}$
such that $\mu_{k,R}\rightarrow \mu_k$ as $R\rightarrow\infty$ in
the sense that for every bounded domain $D$, as
$R\rightarrow\infty$, $\hat{I}_D\mu_{k,R}\rightarrow
\hat{I}_D\mu_k$. Furthermore, as $\eta\rightarrow 0$, $\mu_k^{\eta}$
converges in law, under the topology induced by metric (\ref{e4}),
to $\mu_k$.
\end{theorem}

We call $\mu_k$ the \emph{scaling limit of $k$-arm IIC}.  In
\cite{34}, the authors constructed \emph{CLE$_{\kappa}$ in the
punctured plane} for $8/3<\kappa\leq 4$.  One can view $\mu_1$ as
CLE$_6$ in the punctured plane.  Note that if one can construct IIC
for the discrete $O(n)$ models, it is expected that the scaling
limit of the IIC is just the corresponding CLE$_{\kappa}$ in the
punctured plane.  In particular, the scaling limit of IIC of the
critical Ising model (which is the $O(1)$ model) is expected to be
CLE$_3$ in the punctured plane.

\begin{remark}
From Camia and Newman's construction of the full-plane CLE$_6$, it
is easy to see that full-plane CLE$_6$ is invariant under scalings,
translations, and rotations.  However, with their construction, the
invariance of full-plane CLE$_6$ under the inversion $z\mapsto 1/z$
turns out to be not obvious to establish.  In \cite{37}, using the
Brownian loop soup, the authors proved the inversion-invariance of
full-plane CLE$_{\kappa}$ for $8/3<\kappa\leq 4$.  In \cite{34}, the
inversion-invariance of CLE$_{\kappa}$ in the punctured plane for
$8/3<\kappa\leq 4$ was also proved.  Hence, we propose the following
conjecture:
\begin{conjecture}
The full-plane CLE$_6$ and $\mu_k$ ($k=1,2,4$) are invariant under
$z\mapsto 1/z$.
\end{conjecture}
\end{remark}

\subsubsection{Winding numbers of the arms}

For a curve $\gamma[0,T]$ in the plane with $\gamma(t)\neq 0$ for
all $0\leq t\leq T$, we define the \emph{winding number} of $\gamma$
(around $0$) by $\theta(\gamma):=\arg(\gamma(T))-\arg(\gamma(0))$,
with $\arg$ chosen continuous along $\gamma$.

Denote by $\mathcal {A}^{\eta}$ the event that the percolation
exploration path $\gamma^{\eta}_{\mathbb{D},1,-1}$ intersect with
the boundary of the hexagon $\eta H_0$.  Note that $\mathcal
{A}^{\eta}$ is the same as the event that there is a blue arm
connecting $\eta H_0$ to $\partial_{1,-1}\mathbb{D}^{\eta}$ and a
yellow arm connecting $\eta H_0$ to
$\partial_{-1,1}\mathbb{D}^{\eta}$.

Assume $\mathcal {A}^{\eta}$ occurs and $T$ is the first hitting
time with $\eta H_0$ of $\gamma^{\eta}_{\mathbb{D},1,-1}$. Let
$\theta_{\eta}:=\theta(\gamma^{\eta}_{\mathbb{D},1,-1}[0,T])$.

Theorem \ref{t11} establishes a particular case of Wieland and
Wilson's conjecture on winding number variance of Fortuin-Kasteleyn
contours \cite{50}.

\begin{theorem}\label{t11}
Conditioned on the event $\mathcal {A}^{\eta}$, we have
\begin{equation}\label{e29}
Var[\theta_{\eta}]=\left(\frac{3}{2}+o(1)\right)\log\left(\frac{1}{\eta}\right)\mbox{
as }\eta\rightarrow 0.
\end{equation}
Furthermore, under the conditional measure $P[\cdot|\mathcal
{A}^{\eta}]$,
\begin{equation}\label{e30}
\frac{\theta_{\eta}}{\sqrt{\frac{3}{2}\log\left(\frac{1}{\eta}\right)}}\rightarrow_{d}N(0,1)\mbox{
as }\eta\rightarrow 0.
\end{equation}
\end{theorem}

Suppose the 2-arm event $\mathcal {A}^{\eta}_2(\eta,1)$ happens.  We
fix a deterministic way to choose a unique blue arm connecting
$\partial \mathbb{D}^{\eta}$ and $\eta H_0$, and denote by
$\tilde{\theta}_{\eta}$ the winding number of this arm (here we
consider the arm as a continuous curve by connecting the neighbor
sites with line segments).

The following corollary refines \cite{24} for the 2-arm case by
giving variance estimates and CLT for winding numbers of the arms in
explicit expressions.

\begin{corollary}\label{c4}
Under the conditional measure $P[\cdot|\mathcal
{A}^{\eta}_2(\eta,1)]$ and the 2-arm IIC measure $\nu_2^{\eta}$, as
$\eta\rightarrow 0$, we both have
\begin{equation*}
Var\left[\tilde{\theta}_{\eta}\right]=\left(\frac{3}{2}+o(1)\right)\log\left(\frac{1}{\eta}\right)~~\mbox{
and
}~~\frac{\tilde{\theta}_{\eta}}{\sqrt{\frac{3}{2}\log\left(\frac{1}{\eta}\right)}}\rightarrow_{d}N(0,1).
\end{equation*}
\end{corollary}

\begin{remark}
Corollary \ref{c4} confirms a prediction of Beffara and Nolin
\cite{2} for the 2-arm case explicitly.  Following \cite{24} (see
Theorem 1.1 and Remark 1.2 in \cite{24}), we give the following
conjecture for the 4-arm case:
\begin{conjecture}
Under $P[\cdot|\mathcal {A}^{\eta}_4(\eta,1)]$ and $\nu_4^{\eta}$,
as $\eta\rightarrow 0$ we both have
\begin{equation*}
Var\left[\tilde{\theta}_{\eta}\right]=\left(\frac{3}{8}+o(1)\right)\log\left(\frac{1}{\eta}\right)~~\mbox{
and }~~
\frac{\tilde{\theta}_{\eta}}{\sqrt{\frac{3}{8}\log\left(\frac{1}{\eta}\right)}}\rightarrow_{d}N(0,1).
\end{equation*}
\end{conjecture}
\end{remark}

\begin{remark}
If one can generalize the results for two-sided radial SLE that we
used in this paper to ``$2k$-sided radial SLE", it is expected that
one can use our method to get precise estimate of the winding number
variance for the $2k$-arm case, and get the corresponding CLT.
\end{remark}

\subsubsection{Ideas of the proofs}

Let us explain the main ideas in the proofs of our main results.

\emph{Scaling limits.}  First, we use the approach of
Aizenman-Burchard \cite{33} to show that the $k$-arm IIC has
subsequential scaling limit.  Then, conditioned on the $k$-arm
events for a sequence of annuli, we introduce conditional measures
for percolation and CLE$_6$.  Using these measures, by coupling
argument introduced in \cite{5} and Theorem \ref{t2}, we establish
the uniqueness of the scaling limit.  The conformal invariance of
the scaling limit can be derived from that of CLE$_6$ easily.

\emph{Winding numbers.}  The proof can be divided into three main
steps as follows.
\begin{itemize}
\item First, we use the approach of Schramm \cite{47}
to derive the winding number variance of two-sided radial SLE$_6$.
\item Second, conditioned on the event that the percolation
exploration path in $\mathbb{D}^\eta$ goes through $\eta H_0$, we
show the scaling limit of the path is two-sided radial SLE$_6$.  The
key ingredients include a proposition of Green's function for
chordal SLE proved by Lawler and Rezaei \cite{51}, the coupling
argument and the well-know result that the scaling limit of
percolation exploration path is SLE$_6$.
\item Third, we divide the unit
disk into concentric annuli with large modulus, and show that the
sum of winding number variances of the paths in these annuli
approximates the variance of $\theta_{\eta}$, and the winding number
variance corresponding to each annulus can be approximated well by
that of two-sided radial SLE$_6$ as $\eta\rightarrow 0$.  This step
involves many technical issues and uses coupling argument
extensively.  A key ingredient is the estimate of winding number
variance of the arms from \cite{24}.
\end{itemize}

\section{Preliminary definitions and results}

\subsection{The space of curves}\label{s21}

When taking the scaling limit of percolation on the whole plane, it
is convenient to compactify $\mathbb{C}$ into
$\hat{\mathbb{C}}:=\mathbb{C}\cup\{\infty\}\simeq \mathbb{S}^2$
(i.e., the Riemann sphere) as follows.  First, we replace the
Euclidean metric with a distance function $\Delta(\cdot,\cdot)$
defined on $\mathbb{C}\times\mathbb{C}$ by
\begin{equation}\label{e9}
\Delta(u,v):=\inf_{\varphi}\int(1+|\varphi|^2)^{-1}ds,
\end{equation}
where the infimum is over all smooth curves $\varphi(s)$ joining $u$
with $v$, parameterized by arclength $s$, and $|\cdot|$ denotes the
Euclidean norm.  This metric is equivalent to the Euclidean metric
in bounded regions.  Then, we add a single point $\infty$ at
infinity to get the compact space $\hat{\mathbb{C}}$ which is
isometric, via stereographic projection, to the two-dimensional
sphere.

Let $D$ be a Jordan domain and denote by $\mathcal {S}_{D}$ the
complete separable metric space of continuous curves in
$\overline{D}$ with the metric (\ref{e1}) defined below.  Curves are
regarded as equivalence classes of continuous functions from the
unit interval to $\overline{D}$, modulo monotonic
reparametrizations.  $\mathcal {F}$ will represent a set of curves
(more precisely, a closed subset of $\mathcal {S}_{D}$).
$\textrm{d}(\cdot,\cdot)$ will denote the uniform metric on curves,
defined by
\begin{equation}\label{e1}
\textrm{d}(\gamma_1,\gamma_2):=\inf\sup_{t\in[0,1]}|\gamma_1(t)-\gamma_2(t)|,
\end{equation}
where the infimum is over all choices of parametrizations of
$\gamma_1$ and $\gamma_2$ from the interval $[0,1]$.  The distance
between two closed sets of curves is defined by the induced
Hausdorff metric as follows:
\begin{equation}\label{e2}
\dist(\mathcal {F},\mathcal {F}'):=\inf\{\epsilon>0:\forall
\gamma\in \mathcal {F},\exists \gamma'\in\mathcal {F}'\mbox{ such
that }\textrm{d}(\gamma,\gamma')\leq\epsilon\mbox{ and vice
versa}\}.
\end{equation}
The space $\Omega_D$ of closed subsets of $\mathcal {S}_{D}$ (i.e.,
collections of curves in $\overline{D}$) with the metric (\ref{e2})
is also a complete separable metric space.  Write
$\Omega_R:=\Omega_{\mathbb{D}_R}$.

We will also consider the complete separable metric space $\mathcal
{S}$ of continuous curves in $\hat{\mathbb{C}}$ with the distance
\begin{equation}\label{e3}
\textrm{D}(\gamma_1,\gamma_2):=\inf\sup_{t\in[0,1]}\Delta(\gamma_1(t),\gamma_2(t)),
\end{equation}
where the infimum is again over all choices of parametrizations of
$\gamma_1$ and $\gamma_2$ from the interval $[0,1]$.  The distance
between two closed sets of curves is again defined by the induced
Hausdorff metric as follows:
\begin{equation}\label{e4}
\textrm{Dist}(\mathcal {F},\mathcal {F}'):=\inf\{\epsilon>0:\forall
\gamma\in \mathcal {F},\exists \gamma'\in\mathcal {F}'\mbox{ such
that }\textrm{D}(\gamma,\gamma')\leq\epsilon\mbox{ and vice
versa}\}.
\end{equation}
The space $\Omega$ of closed sets of $\mathcal {S}$ (i.e.,
collections of curves in $\hat{\mathbb{C}}$) with the metric
(\ref{e4}) is also a complete separable metric space.

It was noted in \cite{3,32} that one should add a ``trivial" loop
for each $z$ in $\overline{D}$, so that the collection of CLE$_6$
loops is closed in the appropriate sense \cite{33}.  When
considering the CLE$_6$ in $\hat{\mathbb{C}}$, one should also add a
trivial loop for each $z\in \hat{\mathbb{C}}$ to make the space of
loops closed.  In this paper, we will not include these trivial
loops to the loop process except for dealing with this technical
problem.

\subsection{Coupling argument}\label{s22}
The coupling argument for 1-arm events appeared in \cite{16} for the
construction of IIC, and then the coupling argument for multi-arm
events appeared in \cite{4} for the construction of multi-arm IIC.
Recently, Garban, Pete and Schramm \cite{5} introduced the notion of
faces, and gave the coupling argument in a clear and general form,
which turns out to be very useful.  For example, we used it in
\cite{24} to prove a CLT for the winding numbers of the arms with
alternating colors.  In this paper, we will make extensive use of
coupling argument.  Being familiar with it in \cite{5} and Lemma 2.3
in \cite{24} will be helpful to the readers. First, let us state the
coupling argument that will be used in Section \ref{s3} for $k$-arm
IIC.  To state the result, we need some definitions.

Let $k$ be an even number.  For a circuit $\mathcal
{C}=\gamma_1\gamma_2\ldots\gamma_k$ (i.e., the concatenation of
$\gamma_1,\ldots,\gamma_k$), if $\gamma_1,\ldots,\gamma_k$ are
monochromatic paths with alternating colors, we call $\mathcal {C}$
a \emph{$k$-circuit}, and write $\mathcal
{C}=(\gamma_1,\ldots,\gamma_k)$.  We will always assume that
$\gamma_1$ is blue.  For convenience, a monochromatic blue circuit
is called a \emph{1-circuit}.  For any 4-circuit $\mathcal
{C}=(\gamma_1,\ldots,\gamma_4)$, denote by $U=U_{\mathcal {C}}$ the
indicator function of the event that there exists a blue path
connecting $\gamma_1$ and $\gamma_3$ in $\overline{\mathcal {C}}$
(recall that $\overline{\mathcal {C}}=\mathcal {C}\cup\mbox{
interior sites of }\mathcal {C}$).  Note that $U_{\mathcal {C}}=0$
if and only if there exists a yellow path connecting $\gamma_2$ and
$\gamma_4$ in $\overline{\mathcal {C}}$.

The proofs of the following coupling arguments (which are different
versions of the coupling arguments in \cite{5}) are essentially the
same as those of Proposition 3.1, 3.6 and 5.2 in \cite{5} (see also
the sketch of the proof of Lemma 2.3 in \cite{24}), we omit the
proofs of Proposition \ref{l2} and \ref{l3} except just stating how
to deal with an additional issue in the case of $k=4$ in Proposition
\ref{l2}.

\begin{proposition}\label{l2}
Let $k=1,2,4$. There exists a constant $\alpha=\alpha(k)>0$, such
that for any $10\eta<r<R/100$ and $2r\leq r'\leq R$, there is a
coupling of the measures $P[\cdot|\mathcal {A}_k^{\eta}(\eta,R)]$
and $P[\cdot|\mathcal {A}_k^{\eta}(r,R)]$, such that with
probability at least $1-(r/r')^{\alpha}$ there exists an identical
$k$-circuit $\mathcal {C}$ surrounding the origin in
$A^{\eta}(r,r')$ for both measures, and the configuration outside
$\mathcal {C}$ is also identical, and furthermore $U_{\mathcal {C}}$
is identical in the case $k=4$.
\end{proposition}

\begin{proof}
As we have said before Proposition \ref{l2}, we only deal with the
additional issue for $U_{\mathcal {C}}$ in the case $k=4$. Similarly
to the proofs of Proposition 3.6 in \cite{5} and Lemma 2.3 in
\cite{24}, one can construct a coupling of the measures
$P[\cdot|\mathcal {A}_4^{\eta}(\eta,R)]$ and $P[\cdot|\mathcal
{A}_4^{\eta}(r,R)]$, such that with probability at least
$1-(r/r')^{\alpha}$ the following event $\mathcal {B}$ occurs: There
exists an identical $k$-circuit $\mathcal {C}$ surrounding the
origin in $A^{\eta}(r,r')$ for both measures, and the configuration
outside $\mathcal {C}$ is also identical.  Denote by $\mathcal
{C}_1$ (resp. $\mathcal {C}_2$) the $k$-circuit $\mathcal {C}$ under
$P[\cdot|\mathcal {A}_4^{\eta}(\eta,R)]$ (resp. $P[\cdot|\mathcal
{A}_4^{\eta}(r,R)]$).  Further, the above construction is symmetric
for the colors, so conditioned on $\mathcal {B}$,  the probability
of $U_{\mathcal {C}_2}=1$ equals to that of $U_{\mathcal {C}_2}=0$.
Note that the color of the hexagon $\eta H_0$ under
$P[\cdot|\mathcal {A}_4^{\eta}(\eta,R)]$ is essentially irrelevant
to the construction of the coupling.  Hence, if $\mathcal {B}$
occurs, one can let the hexagon be blue if $U_{\mathcal {C}_2}=1$,
and yellow if $U_{\mathcal {C}_2}=0$; otherwise we toss a coin to
determine the color.  Then under this new coupling one has
$U_{\mathcal {C}_1}=U_{\mathcal {C}_2}$.
\end{proof}

\begin{proposition}\label{l3}
Let $k=1,2,4$.  There exists a constant $\alpha=\alpha(k)>0$, such
that for any $100\eta<R_1<R_2$ and $10\eta<r<R_1/2$, there is a
coupling of $P[\cdot|\mathcal {A}_k^{\eta}(\eta,R_1)]$ and
$P[\cdot|\mathcal {A}_k^{\eta}(\eta,R_2)]$, so that with probability
at least $1-(r/R_1)^{\alpha}$ there exists an identical $k$-circuit
$\mathcal {C}$ surrounding the origin in $A^{\eta}(r,R_1)$ for both
measures, and the configuration inside $\mathcal {C}$ is also
identical.
\end{proposition}

Now, we want to give the coupling argument that will be used in
Section \ref{s4} for winding numbers.  Following the terminology of
\cite{5,24}, we first introduce the notion of faces.  Let $x_1,x_2$
be distinct e-vertices in $\partial\mathbb{D}_R^{\eta}$. Let
$\gamma_1$ be a blue path of hexagons joining $x_1$ to $x_2$ and let
$\gamma_2$ be a yellow path of hexagons joining $x_2$ to $x_1$.
Denote by $\Theta=(\gamma_1,\gamma_2)$ the circuit which is composed
of the two paths.  We assume furthermore that
$\mathbb{D}_R^{\eta}\subset\mbox{ interior of }\Theta$. Then we call
the circuit $\Theta$ a configuration of \emph{faces} with endpoints
$x_1,x_2$, and say $\Theta$ are faces around $\partial
\mathbb{D}_R^{\eta}$.  Define the \emph{quality} of a configuration
of faces $Q(\Theta)$ to be the distance between the endpoints,
normalized by $R$. That is,
$$Q(\Theta):=\frac{|x_1-x_2|}{R}.$$

Let $Cone_1:=\{z\in\mathbb{C}:-3\pi/4<\arg(z)<3\pi/4\}$,
$Cone_2:=\{z\in\mathbb{C}:\:\pi/4<\arg(z)<7\pi/4\}$,
$Cone_3:=\{z\in\mathbb{C}:-\pi/4<\arg(z)<\pi/4\}$,
$Cone_4:=\{z\in\mathbb{C}:\:3\pi/4<\arg(z)<5\pi/4\}$.

In the annulus $A=A^{\eta}(R,2R)$, let $\mathcal {R}=\mathcal
{R}(A)$ be the event that there are exactly two disjoint alternating
arms crossing $A$, and the resulting two interfaces are contained
respectively in $Cone_1$ and $Cone_2$, with the endpoints of the
interfaces on the two boundaries of $A$ belonging to $Cone_3$ and
$Cone_4$, respectively.

Lemma \ref{l21} is the straightforward 2-arm analog of Lemma 2.2 in
\cite{24}.  The proof is analogous to the second proof of Lemma 3.4
in \cite{5}, we leave it to the reader.
\begin{lemma}\label{l21}
$P(\mathcal {R}(A^{\eta}(R,2R)))>C$ for an absolute constant $C>0$.
\end{lemma}

\begin{figure}
\begin{center}
\includegraphics[height=0.4\textwidth]{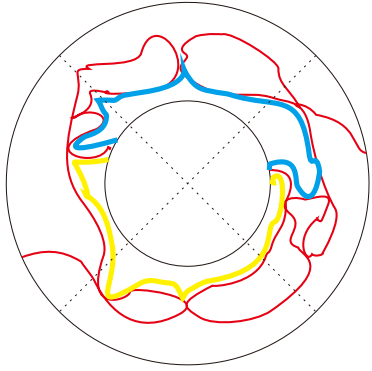}
\caption{Two interfaces crossing the annulus induce a natural
configuration of good faces.}\label{fig1}
\end{center}
\end{figure}

For $A^{\eta}(R,2R)$, if the event $\mathcal {R}$ happens, then the
two interfaces induce a natural configuration of faces
$\Theta\subset A^{\eta}(R,2R)$ around $\partial
\mathbb{D}_R^{\eta}$.  We call $\Theta$ \emph{good faces} around
$\partial \mathbb{D}_R^{\eta}$.  See Figure \ref{fig1}.

For $\eta\leq r<R$ and faces $\Theta=(\gamma_1,\gamma_2)$ around
$\partial \mathbb{D}_R^{\eta}$, define
\begin{equation*}
\mathcal {A}_{\Theta}^{\eta}(r,R):=\left\{
\begin{aligned}
&\mbox{$\exists$ a blue arm
connecting $\gamma_1$ to $\partial \mathbb{D}_r^{\eta}$ and}\\
&\mbox{a yellow arm connecting $\gamma_2$ to $\partial
\mathbb{D}_r^{\eta}$}
\end{aligned}
\right\}.
\end{equation*}

Let us now define a measure $P_R^*[\cdot]$ as follows.  First, we
sample good faces $\Theta$ around $\partial \mathbb{D}_R^{\eta}$
according to the law $P[\cdot|\mathcal {R}]$; then conditioning on
$\Theta$, we sample the configuration inside $\Theta$ according to
$P[\cdot|\mathcal {A}_{\Theta}(\eta,R)]$.  This induces a
probability measure on good faces around $\partial
\mathbb{D}_R^{\eta}$ and the configuration inside the good faces,
denoted by $P_R^*$.

For $\eta<R$, denote by $\mathcal {A}^{\eta}(R)$ the event that the
percolation exploration path $\gamma_{\mathbb{D}_R,R,-R}^{\eta}$
intersects with the hexagon $\eta H_0$.  Note that $\mathcal
{A}^{\eta}(1)=\mathcal {A}^{\eta}$.

The proofs of the following coupling results are very similar to
those of Proposition 3.1 and 3.6 in \cite{5} (see also Lemma 2.3 in
\cite{24}), which are omitted here.
\begin{proposition}\label{p3}
There exists a constant $\beta>0$, such that for all
$\eta<1/100,10\eta<r<R/2$ and $R\leq 1$, there is a coupling of the
measures $P[\cdot|\mathcal {A}^{\eta}]$, $P[\cdot|\mathcal
{A}^{\eta}(R)]$ and $P[\cdot|\mathcal {A}_2^{\eta}(\eta,R)]$, so
that with probability at least $1-(r/R)^{\beta}$ there exist
identical good faces $\Theta\subset A^{\eta}(r,R)$ for these three
measures, and the configuration in $\overline{\Theta}$ is also
identical.
\end{proposition}

\begin{proposition}\label{p4}
There exist constants $C_0,C_1>0$, such that for all $\eta\leq
r<R/2$, any fixed faces $\Theta$ around $\partial
\mathbb{D}_R^{\eta}$ and $N:=\lfloor\log_2(R/r)\rfloor$, there is a
coupling of $P[\cdot|\mathcal {A}_{\Theta}^{\eta}(r,R)]$ and
$\{P_{(1/2)^jR}^*[\cdot]\}_{1\leq j\leq N}$, so that
\begin{itemize}
\item for all
$1\leq j\leq N $, with probability at least $1-\exp(-C_0j)$, there
exists $1\leq j^*\leq j$ such that there exist good faces
$\Theta_{j^*}$ around $\partial \mathbb{D}^{\eta}_{(1/2)^{j^*}R}$
under $P[\cdot|\mathcal {A}_{\Theta}^{\eta}(r,R)]$, and the
configuration in $\overline{\Theta}_{j^*}$ under $P[\cdot|\mathcal
{A}_{\Theta}^{\eta}(r,R)]$ is the same as the configuration under
$P_{(1/2)^{j^*}R}^*[\cdot]$;
\item for all $1\leq j\leq N-1$, with probability at least
$\exp(-C_1(j+1))$, for all $1\leq j'\leq j$ there do not exist good
faces around $\partial \mathbb{D}^{\eta}_{(1/2)^{j'}R}$, but there
exist good faces $\Theta_{j+1}$ around $\partial
\mathbb{D}^{\eta}_{(1/2)^{j+1}R}$ under $P[\cdot|\mathcal
{A}_{\Theta}^{\eta}(r,R)]$, and the configuration in
$\overline{\Theta}_{j+1}$ under $P[\cdot|\mathcal
{A}_{\Theta}^{\eta}(r,R)]$ is the same as the configuration under
$P_{(1/2)^{j+1}R}^*[\cdot]$.
\end{itemize}
\end{proposition}

For the next proposition we need some additional notation.  Let
$0<r<1$.  For the percolation exploration path
$\gamma^{\eta}_{\mathbb{D},1,-1}$,  define event
\begin{equation*}
\mathcal {A}_{r}^{\eta}:=\{\gamma_{\mathbb{D},1,-1}^{\eta}\cap
\partial \mathbb{D}_r^{\eta}\neq\emptyset\}.
\end{equation*}
For a curve $\gamma$ with $\gamma\cap\partial
\mathbb{D}_r^{\eta}\neq \emptyset$, denote by $\tau_r^{\eta}$ the
first hitting time with $\partial \mathbb{D}_r^{\eta}$ of $\gamma$.
\begin{proposition}\label{p5}
There exists a constant $\beta>0$, such that for all
$100\eta<10r<R<1$, there is a coupling of the measures
$P[\cdot|\mathcal {A}^{\eta}]$ and $P[\cdot|\mathcal {A}_r^{\eta}]$,
so that with probability at lest $1-(r/R)^{\beta}$, the stopped
percolation exploration path
$\gamma^{\eta}_{\mathbb{D},1,-1}[0,\tau^{\eta}_R]$ under
$P[\cdot|\mathcal {A}^{\eta}]$ is identical to that under
$P[\cdot|\mathcal {A}_r^{\eta}]$.
\end{proposition}

\subsection{Basic properties of arm events}\label{s23}

In this paper, we assume that the reader is familiar with the FKG
inequality (see Lemma 13 in \cite{19} for generalized FKG), the BK
(van den Berg-Kesten) inequality and Reimer's inequality \cite{45},
and the RSW (Russo-Seymour-Welsh) technology. See \cite{9,23}.  The
following properties of arm events are well known (see \cite{19})
except (\ref{e72}) and (\ref{e67}), where (\ref{e67}) is a
generalization of the standard quasi-multiplicativity.

\begin{enumerate}
\item \emph{A priori bounds for arm events}:  For any color sequence $\sigma$, there
exist $C_1(|\sigma|)$, $C_2(|\sigma|)$, $\alpha(|\sigma|)$,
$\beta(|\sigma|)>0$ such that for all $\eta\leq r<R$,
\begin{equation}\label{e10}
C_1\left(\frac{r}{R}\right)^{\alpha}\leq P[\mathcal
{A}^{\eta}_{\sigma}(r,R)]\leq C_2\left(\frac{r}{R}\right)^{\beta}.
\end{equation}
\item There exists a constant $C>0$, such that for all $\eta\leq r<R$ and faces $\Theta$ around $\partial
\mathbb{D}^{\eta}_R$ with $Q(\Theta)>1/4$,
\begin{equation}\label{e72}
CP[\mathcal {A}_2^{\eta}(r,R)]\leq P[\mathcal
{A}_\Theta^{\eta}(r,R)]\leq P[\mathcal {A}_2^{\eta}(r,R)].
\end{equation}
\item \emph{Quasi-multiplicativity}: For any color sequence $\sigma$, there
is a $C_1(|\sigma|)>0$, such that for all $\eta\leq r_1<r_2\leq
r_3<r_4$ and $r_3\leq 10r_2$,
\begin{equation*}
C_1P[\mathcal {A}_{\sigma}^{\eta}(r_1,r_2)]P[\mathcal
{A}_{\sigma}^{\eta}(r_3,r_4)]\leq P[\mathcal
{A}_{\sigma}^{\eta}(r_1,r_4)]\leq P[\mathcal
{A}_{\sigma}^{\eta}(r_1,r_2)]P[\mathcal
{A}_{\sigma}^{\eta}(r_3,r_4)].
\end{equation*}
Furthermore,  there is a $C_2>0$, such that for all $\eta\leq
r_1<r_2\leq r_3/2$ and any given faces $\Theta$ around $\partial
\mathbb{D}^{\eta}_{r_3}$,
\begin{equation}\label{e67}
C_2P[\mathcal {A}_2^{\eta}(r_1,r_2)]P[\mathcal
{A}_{\Theta}^{\eta}(r_2,r_3)]\leq P[\mathcal
{A}_{\Theta}^{\eta}(r_1,r_3)]\leq P[\mathcal
{A}_2^{\eta}(r_1,r_2)]P[\mathcal {A}_{\Theta}^{\eta}(r_2,r_3)].
\end{equation}
\end{enumerate}

\begin{proof}
We just need to prove (\ref{e72}) and (\ref{e67}). Applying a
standard gluing argument with generalized FKG, RSW and Theorem 11 in
\cite{19}, one gets (\ref{e72}). The details are omitted.  Now let
us show (\ref{e67}).   Conditioned on $\mathcal
{A}_{\Theta}^{\eta}(r_2,r_3)$, the two interfaces (or b-paths)
starting from the endpoints of $\Theta=(\gamma_1,\gamma_2)$ to reach
$\partial \mathbb{D}_{2r_3/3}^{\eta}$ together with $\Theta$ induce
faces $\Theta'=(\gamma_1',\gamma_2')$ around $\partial
\mathbb{D}_{2r_3/3}^{\eta}$.  By Lemma 3.3 (Strong Separation Lemma)
in \cite{5}, there is some absolute constant $C_3>0$ such that
\begin{equation}\label{e68}
P\left[Q(\Theta')>\frac{1}{4}\mid\mathcal
{A}_{\Theta}^{\eta}(r_2,r_3)\right]\geq C_3.
\end{equation}
By a gluing construction with FKG, RSW and Theorem 11 in \cite{19},
there is some absolute constant $C_4>0$ such that for any given
$\Theta'$ with $Q(\Theta')>1/4$ (see an analogous
quasi-multiplicativity in \cite{24}),
\begin{equation}\label{e69}
C_4P[\mathcal {A}_2^{\eta}(r_1,r_2)]P[\mathcal
{A}_{\Theta'}^{\eta}(r_2,2r_3/3)]\leq P[\mathcal
{A}_{\Theta'}^{\eta}(r_1,2r_3/3)].
\end{equation}
Define
\begin{equation*}
\mathcal {A}_{\Theta,\Theta'}:=\left\{
\begin{aligned}
&\mbox{$\exists$ a blue arm connecting $\gamma_1$ and $\gamma_1'$ and}\\
&\mbox{a yellow arm connecting $\gamma_2$ and $\gamma_2'$}
\end{aligned}
\right\}.
\end{equation*}
By (\ref{e68}) and (\ref{e69}), we have
\begin{align*}
C_3C_4P[\mathcal {A}_2^{\eta}&(r_1,r_2)]P[\mathcal {A}_{\Theta}^{\eta}(r_2,r_3)]\\
&\leq C_4\sum_{Q(\Theta')>1/4}P[\Theta',\mathcal
{A}_{\Theta,\Theta'}]P[\mathcal
{A}_{\Theta'}^{\eta}(r_2,2r_3/3)]P[\mathcal
{A}_2^{\eta}(r_1,r_2)]\\
&\leq \sum_{Q(\Theta')>1/4}P[\Theta',\mathcal
{A}_{\Theta,\Theta'}]P[\mathcal
{A}_{\Theta'}^{\eta}(r_1,2r_3/3)]\leq P[\mathcal
{A}_{\Theta}^{\eta}(r_1,r_3)].
\end{align*}
By choosing $C_2=C_3C_4$, we conclude the proof.
\end{proof}

\section{Scaling limit of multi-arm IIC}\label{s3}
In this section we will prove our main results concerning the
scaling limit of $k$-arm IIC.  First we give some lemmas that will
be used. The following lemma can be seen as an analog of Lemma 2.9
in \cite{5} for quad-crossing percolation limit.

\begin{lemma}\label{l11}
For any $0<r<R$ and $k=1,2,4$, there exists a constant $C_k>0$
(depending on $r/R$), such that
\begin{equation}\label{e15}
\lim_{\eta\rightarrow 0}\mu_R^{\eta}[\mathcal
{A}_k^{\eta}(r,R)]=\mu_R[\mathcal {A}_k(r,R)]>C_k.
\end{equation}
Moreover, in any coupling of the measures $\{\mu_R^{\eta}\}$ and
$\mu_R$ on $(\Omega_R,\mathcal {F}_R)$ in which
$\dist(\omega_R^{\eta},\omega_R)\rightarrow 0$ a.s. as
$\eta\rightarrow 0$, we have
\begin{equation}\label{e14}
\hat{P}[\{\omega_R^{\eta}\in \mathcal
{A}_k^{\eta}(r,R)\}\Delta\{\omega_R\in \mathcal
{A}_k(r,R)\}]\rightarrow 0\mbox{ as }\eta\rightarrow 0,
\end{equation}
where $\hat{P}[\cdot]$ denotes the coupling measure.
\end{lemma}

\begin{proof}
By Theorem \ref{t2}, we can couple the measures $\{\mu_R^{\eta}\}$
and $\mu_R$ on $(\Omega_R,\mathcal {F}_R)$ such that
$\dist(\omega_R^{\eta},\omega_R)\rightarrow 0$ a.s. as
$\eta\rightarrow 0$.  Let us show (\ref{e14}) for $k=1,2,4$
respectively in the following.

By (\ref{e11}) and the definition of $\mathcal {A}_1(r,R)$,  it is
easy to see that for each small $\epsilon>0$ and $\eta<\epsilon$,
\begin{align*}
&\hat{P}[\{\omega_R^{\eta}\in \mathcal
{A}_1^{\eta}(r,R)\}\Delta\{\omega_R\in \mathcal {A}_1(r,R)\}]\\
&\leq\hat{P}[\dist(\omega_R^{\eta},\omega_R)\geq\epsilon]+\hat{P}\left[
\begin{aligned}
&\mbox{$\exists$ counterclockwise loop
$\gamma^{\eta}\in\omega_R^{\eta}$
surrounding the}\\
&\mbox{origin in $A(r-\epsilon,R)$, and $\gamma^{\eta}\cap
A(r-\epsilon,r+\epsilon)\neq\emptyset$}
\end{aligned}
\right].
\end{align*}
The first term goes to zero as $\eta\rightarrow 0$.  The event in
the second term produces a half-plane 3-arm event from the
$2\epsilon$-neighborhood of $\partial\mathbb{D}_r$ to a distance of
unit order, whose probability goes to zero as $\epsilon\rightarrow
0$, since the polychromatic half-plane 3-arm exponent is 2; see,
e.g., Lemma 6.8 in \cite{35}.  Then (\ref{e14}) is proved in the
case $k=1$.

By (\ref{e12}) and the definition of $\mathcal {A}_2(r,R)$, for each
small $\epsilon>0$ and $\eta<\epsilon$,
\begin{align*}
&\hat{P}[\{\omega_R\in \mathcal
{A}_2(r,R)\}\backslash\{\omega_R^{\eta}\in \mathcal
{A}_2^{\eta}(r,R)\}]\\
&\leq\hat{P}[\dist(\omega_R^{\eta},\omega_R)\geq\epsilon]+\hat{P}\left[
\begin{aligned}
&\mbox{$\exists$ counterclockwise loop
$\gamma^{\eta}\in\omega_R^{\eta}$
intersecting with $\partial \mathbb{D}_{r+\epsilon}$}\\
&\mbox{and $\partial \mathbb{D}_{R-\epsilon}$ in
$\overline{\mathbb{D}_R^{\eta}}$, and
$\gamma^{\eta}\cap\partial\mathbb{D}_r^{\eta}=\emptyset$ or
$\gamma^{\eta}\cap\partial\mathbb{D}_R^{\eta}=\emptyset$}
\end{aligned}
\right].
\end{align*}
The event in the second term implies a half-plane 3-arm event from
the $2\epsilon$-neighborhood of $\partial\mathbb{D}_r$ or
$\partial\mathbb{D}_R$ to a distance of unit order, whose
probability goes to zero as $\epsilon\rightarrow 0$.  Then we get
that $\hat{P}[\{\omega_R\in \mathcal
{A}_2(r,R)\}\backslash\{\omega_R^{\eta}\in \mathcal
{A}_2^{\eta}(r,R)\}]\rightarrow 0$ as $\eta\rightarrow 0$.  Now let
us show the other direction.  Similarly, for each small $\epsilon>0$
and $\eta<\epsilon$, we have
\begin{align*}
&\hat{P}[\{\omega_R^{\eta}\in \mathcal
{A}_2^{\eta}(r,R)\}\backslash\{\omega_R\in \mathcal
{A}_2(r,R)\}]\\
&\leq\hat{P}[\dist(\omega_R^{\eta},\omega_R)\geq\epsilon]+\hat{P}
\left[
\begin{aligned}
&\mbox{$\exists$ counterclockwise loop $\gamma\in\omega_R$
intersecting with $\partial \mathbb{D}_{r+\epsilon}$}\\
&\mbox{and $\partial \mathbb{D}_{R-2\epsilon}$ in
$\overline{\mathbb{D}}_R$, and
$\gamma\cap\partial\mathbb{D}_r=\emptyset$ or
$\gamma\cap\partial\mathbb{D}_R=\emptyset$}
\end{aligned}
\right].
\end{align*}
Clearly the second term goes to zero as $\epsilon\rightarrow 0$.
Then (\ref{e14}) is proved in the case $k=2$.

Similarly to the case $k=2$, one can prove the case where $k=4$, and
the details are omitted.

(\ref{e10}) and (\ref{e14}) imply (\ref{e15}) immediately.
\end{proof}

A collection of measures is said to be \emph{(weakly) relatively
compact} if every sequence has a convergent subsequence.  To prove
the existence of the scaling limit, we need a lemma on the existence
of subsequential scaling limits:

\begin{lemma}\label{l6}
Let $k=1,2,4$.  $\{\mu_R^\eta[\cdot|\mathcal
{A}_k^{\eta}(\eta,R)]\}_{\eta}$ and $\{\mu_k^{\eta}[\cdot]\}_{\eta}$
are relatively compact.
\end{lemma}

\begin{proof}
We use the machinery developed by Aizenman and Burchard (Theorem 1.2
in \cite{33}).  Let $\mu^{\eta}$ denote the probability measure
supported on collections of curves that are polygonal paths on the
edges of $\eta\mathbb{H}$ in $\overline{\mathbb{D}_R^{\eta}}$. In
our setting, Hypothesis H1 of \cite{33} is as follows.

\textbf{Hypothesis H1}.  For all $j\in\mathbb{N}$, $z\in
\overline{\mathbb{D}}_R$ and $\eta\leq r_1<r_2\leq 1$, the following
bound holds uniformly in $\eta$ and $z$:

\begin{equation*}
\mu^{\eta}[A(z;r_1,r_2)\mbox{ is traversed $j$ times by a
curve}]\leq K_j(r_1/r_2)^{\phi(j)}
\end{equation*}
for some $K_j<\infty$ and $\phi(j)\rightarrow\infty$ as
$j\rightarrow\infty$.

Observe that the number of segments of a loop crossing an annulus is
necessarily even and that, if the annulus is traversed by $j\in
2\mathbb{N}$ separate segments of a loop $\in \omega_R^{\eta}$,
there will be $j/2$ disjoint yellow arms crossing this annulus.  Now
let us prove that $\{\mu_R^\eta[\cdot|\mathcal
{A}_k^{\eta}(\eta,R)]\}_{\eta}$ satisfies Hypothesis H1 for
$k=1,2,4$.  First, we do this in the case $k=1$, which is the
easiest one.

The BK inequality and (\ref{e10}) imply that there exist constants
$C>1,\lambda>0$, such that for all $\eta\leq r_1<r_2$,
$z\in\mathbb{C}$ and $j\in\mathbb{N}$,
\begin{equation}\label{e7}
P_R^{\eta}[\mathcal {A}_{j,Y\ldots
Y}^{\eta}(z;r_1,r_2)]\leq\left\{P_R^{\eta}[\mathcal
{A}_Y^{\eta}(z;r_1,r_2)]\right\}^j\leq C^j(r_1/r_2)^{\lambda j}.
\end{equation}
Let $j\in2\mathbb{N},\eta\leq r_1<r_2,z\in \overline{\mathbb{D}}_R$,
we have
\begin{align*}
&P_R^\eta[A(z;r_1,r_2)\mbox{ is traversed $j$ times by a
loop }|\mathcal {A}_1^{\eta}(\eta,R)]\\
&\leq \frac{P_R^{\eta}\left[\mathcal {A}_1^{\eta}(\eta,R),
\mathcal {A}_{j/2,Y\ldots Y}^{\eta}(z;r_1,r_2)\right]}{P_R^{\eta}[\mathcal {A}_1^{\eta}(\eta,R)]}\\
&\leq P_R^{\eta}\left[
\mathcal {A}_{j/2,Y\ldots Y}^{\eta}(z;r_1,r_2)\right]~~\mbox{by Reimer's inequality}\\
&\leq C^{j/2}(r_1/r_2)^{\lambda j/2}~~\mbox{by (\ref{e7})}.
\end{align*}

Now let us consider the cases of $k=2,4$.  Without loss of
generality, we assume $10\eta\leq 10r_1\leq r_2\leq R/4,j\in
2\mathbb{N}$ and $j\geq k+2$.  Let $C_1,C_2,C_3$ (just depending on
$k$) be appropriate positive constants.  We will distinguish the
following four cases (see Figure \ref{fig2}).

\begin{figure}
\begin{center}
\includegraphics[height=0.7\textwidth]{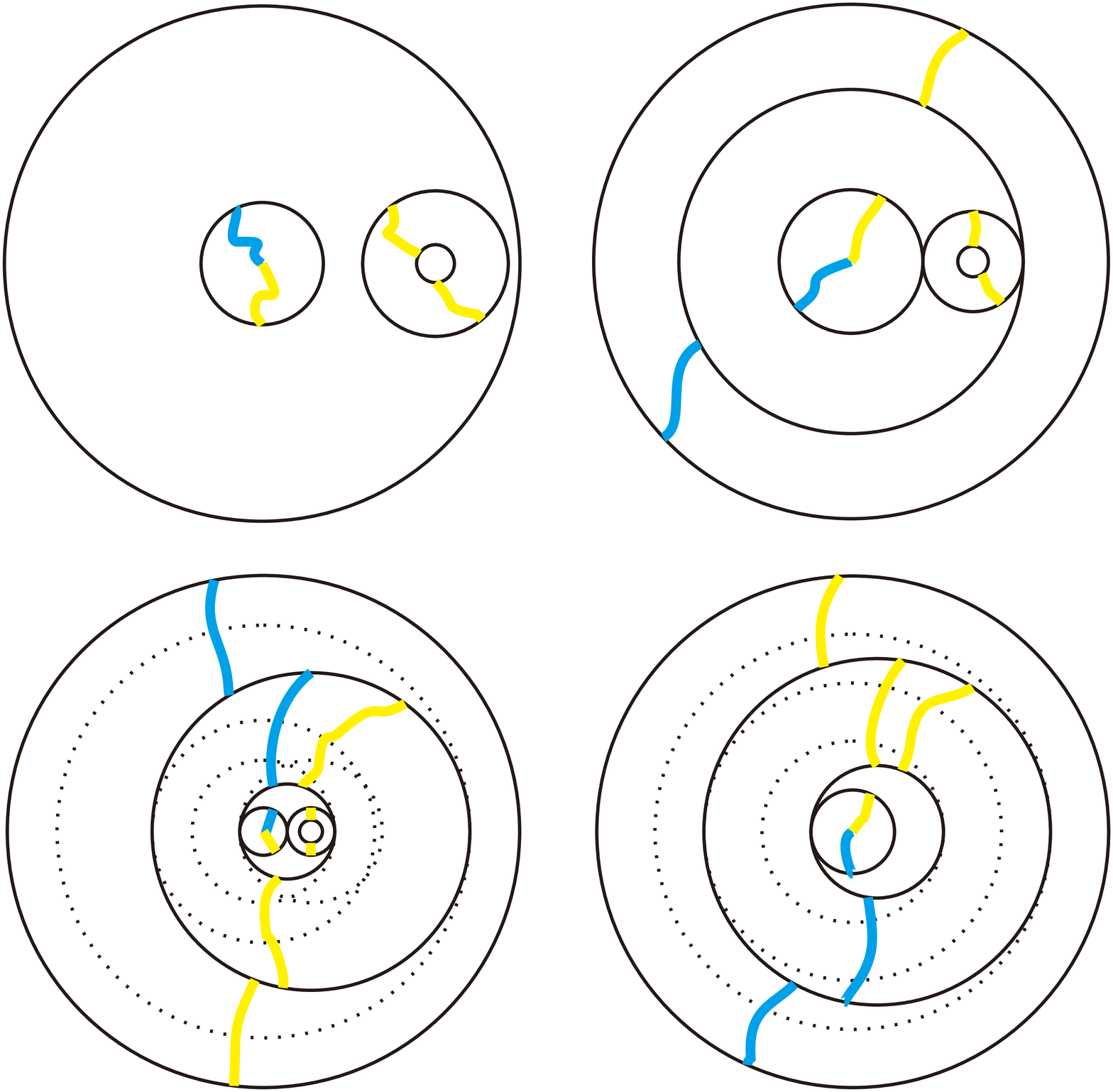}
\caption{A sketch of the four cases in the proof of Lemma \ref{l6}
($k=2,j=4$).}\label{fig2}
\end{center}
\end{figure}

Case 1: $R/2\leq |z|\leq R$.
\begin{align*}
&P_R^\eta[A(z;r_1,r_2)\mbox{ is traversed $j$ times by a loop
}|\mathcal
{A}_k^{\eta}(\eta,R)]\\
&\leq \frac{P_R^{\eta}[\mathcal
{A}_k^{\eta}(\eta,R/5)]}{P_R^{\eta}[\mathcal
{A}_k^{\eta}(\eta,R)]}P_R^{\eta}\left[\mathcal
{A}_{j/2,Y\ldots Y}^{\eta}(z;r_1,r_2)\right]\\
&\leq C_1C^{j/2}(r_1/r_2)^{\lambda j/2}~~\mbox{by
quasi-multiplicativity and (\ref{e7})}.
\end{align*}

Case 2: $r_2/3\leq |z|\leq R/2$.
\begin{align*}
&P_R^\eta[A(z;r_1,r_2)\mbox{ is traversed $j$ times by a loop
}|\mathcal
{A}_k^{\eta}(\eta,R)]\\
&\leq \frac{P_R^{\eta}[\mathcal
{A}_k^{\eta}(\eta,|z|-r_2/4),\mathcal
{A}_k^{\eta}(|z|+r_2/4,R)]}{P_R^{\eta}[\mathcal
{A}_k^{\eta}(\eta,R)]}P_R^{\eta}\left[\mathcal
{A}_{j/2,Y\ldots Y}^{\eta}\left(z;r_1,r_2/4\right)\right]\\
&\leq C_2C^{j/2}(r_1/r_2)^{\lambda j/2}~~\mbox{by
quasi-multiplicativity and (\ref{e7})}.
\end{align*}

Case 3: $3r_1\leq |z|\leq r_2/3$.
\begin{align*}
&P_R^\eta[A(z;r_1,r_2)\mbox{ is traversed $j$ times by a loop
}|\mathcal
{A}_k^{\eta}(\eta,R)]\\
&\leq P_R^{\eta}\left[\mathcal
{A}_k^{\eta}\left(\eta,|z|/2\right),\mathcal
{A}_k^{\eta}\left(\overline{\mathbb{D}(z;r_2)}\backslash\mathbb{D}\left(z/2;|z|\right)\right),\mathcal
{A}_k^{\eta}(\overline{\mathbb{D}}_R\backslash\mathbb{D}(z;r_2)),\right.\\
&~~~~~~~~~~~\left.\mathcal {A}_{j/2,Y\ldots
Y}^{\eta}\left(z;r_1,|z|/2\right),\mathcal {A}_{j/2,Y\ldots
Y}^{\eta}\left(\overline{\mathbb{D}(z;r_2)}\backslash\mathbb{D}\left(z/2;|z|\right)\right)
\right]/P_R^{\eta}[\mathcal
{A}_k^{\eta}(\eta,R)]\\
&\leq \frac{P_R^{\eta}[\mathcal {A}_k^{\eta}(\eta,|z|/2),\mathcal
{A}_k^{\eta}(3|z|/2,r_2-|z|),\mathcal
{A}_k^{\eta}(r_2+|z|,R)]}{P_R^{\eta}[\mathcal
{A}_k^{\eta}(\eta,R)]}\times\\
&~~~~~P_R^{\eta}\left[\mathcal {A}_{j/2,Y\ldots
Y}^{\eta}\left(z;r_1,|z|/2\right),\mathcal {A}_{(j-k)/2,Y\ldots
Y}^{\eta}\left(z;3|z|/2,r_2\right)\right]\mbox{by Reimer's inequality}\\
&\leq C_3C^{(j-k)/2}(r_1/r_2)^{\lambda(j-k)/2}~~\mbox{by
quasi-multiplicativity and (\ref{e7})}.
\end{align*}

Case 4: $|z|\leq 3r_1$.
\begin{align*}
&P_R^\eta[A(z;r_1,r_2)\mbox{ is traversed $j$ times by a loop
}|\mathcal
{A}_k^{\eta}(\eta,R)]\\
&\leq \frac{P_R^{\eta}[\mathcal {A}_k^{\eta}(\eta,r_1),\mathcal
{A}_k^{\eta}(z;|z|+r_1,r_2),\mathcal
{A}_k^{\eta}(\overline{\mathbb{D}}_R\backslash\mathbb{D}(z;r_2)),
\mathcal {A}_{j/2,Y\ldots
Y}^{\eta}(z;|z|+r_1,r_2)]}{P_R^{\eta}[\mathcal
{A}_k^{\eta}(\eta,R)]}\\
&\leq \frac{P_R^{\eta}[\mathcal {A}_k^{\eta}(\eta,r_1),\mathcal
{A}_k^{\eta}(2|z|+r_1,r_2-|z|),\mathcal
{A}_k^{\eta}(r_2+|z|,R)]}{P_R^{\eta}[\mathcal
{A}_k^{\eta}(\eta,R)]}\times\\
&~~~~~~~~P_R^{\eta}\left[\mathcal {A}_{(j-k)/2,Y\ldots
Y}^{\eta}(z;|z|+r_1,r_2)\right]~~\mbox{by Reimer's inequality}\\
&\leq C_4C^{(j-k)/2}(r_1/r_2)^{\lambda(j-k)/2}~~\mbox{by
quasi-multiplicativity and (\ref{e7})}.
\end{align*}
Hence, for $k=1,2,4$, $\{\mu_R^\eta[\cdot|\mathcal
{A}_k^{\eta}(\eta,R)]\}_{\eta}$ satisfies Hypothesis H1.  Then
Theorem 1.2 in \cite{33} implies that it is relatively compact.

For the relatively compactness of $\{\mu_k^{\eta}[\cdot]\}_{\eta}$,
we need to consider $\hat{\mathbb{C}}$ with metric (\ref{e9}).  It
is noted in the Remark just below Theorem 3.1 in \cite{46}, although
Theorem 1.2 in \cite{33} was formulated for compact subsets $\Lambda
\subset \mathbb{R}^d$, it also applies to this case.  By the
inequalities above and the definition of $\mu_k^{\eta}$, we have
that there exists a constant $C_5>0$ depending on $k$, such that
\begin{align*}
&\mu_k^{\eta}(A(z;r_1,r_2)\mbox{ is traversed $j$ times by a loop})\\
&=\lim_{R\rightarrow\infty}P_R^{\eta}[A(z;r_1,r_2)\mbox{ is
traversed $j$ times by a loop }|\mathcal {A}_k^{\eta}(\eta,R)]\leq
C_5C^j(r_1/r_2)^{\lambda(j-k)/2}.
\end{align*}
Similarly to the proof of (i) of Theorem 1.1 in \cite{46}, by Lemma
3.3 in \cite{46}, the corresponding bound on crossing probabilities
holds (with the same exponents) also for the system on
$\hat{\mathbb{C}}$ with the metric (\ref{e9}).  Then Theorem 1.2 in
\cite{33} implies that $\{\mu_k^{\eta}[\cdot]\}_{\eta}$ is
relatively compact for $k=1,2,4$.
\end{proof}

The following lemma is a particular case of the first part of
Theorem \ref{t6}, the proof of the general case is essentially the
same as for this lemma.

\begin{lemma}\label{l4}
Let $k=1,2,4$. For each $R>0$, as $\eta\rightarrow 0$ and
$\epsilon\rightarrow 0$, $\mu_R^\eta[\cdot|\mathcal
{A}_k^{\eta}(\eta,R)]$ and $\mu_R[\cdot|\mathcal {A}_k(\epsilon,R)]$
converge in law, under the topology induced by metric (\ref{e2}), to
the same probability distribution, denoted by $\mu_{k,R}$.
\end{lemma}
\begin{proof}
By Theorem \ref{t2} and Lemma \ref{l11}, for any fixed small
$\epsilon>0$ and $\delta>0$, we can couple
$\mu_R^{\eta}[\cdot|\mathcal {A}_k^{\eta}(\epsilon,R)]$ and
$\mu_R[\cdot|\mathcal {A}_k(\epsilon,R)]$ for all $\eta$ small
enough, such that with probability at least $1-\delta$,
\begin{equation}\label{e5}
\dist(\omega_{k,\epsilon,R}^{\eta},\omega_{k,\epsilon,R})\leq
\delta,
\end{equation}
where $\omega_{k,\epsilon,R}^{\eta},\omega_{k,\epsilon,R}$ are the
configurations under these two laws.

By Proposition \ref{l2}, there exists a constant $\alpha>0$ such
that for small $\epsilon>10\eta$, we can couple
$\mu_R^{\eta}[\cdot|\mathcal {A}_k^{\eta}(\epsilon,R)]$ and
$\mu_R^\eta[\cdot|\mathcal {A}_k^{\eta}(\eta,R)]$, such that with
probability at least $1-\epsilon^{\alpha/2}$, there exists an
identical $k$-circuit $\mathcal {C}=\mathcal {C}(\eta,\epsilon)$
surrounding the origin in $A(\epsilon,\sqrt{\epsilon})$ for both
measures, and the configuration outside $\mathcal {C}$ is also
identical, and furthermore $U_{\mathcal {C}}$ is identical in the
case $k=4$.  Observe that when the above event happens,
\begin{equation}\label{e6}
\dist(\omega_{k,\epsilon,R}^{\eta},\omega_{k,\eta,R}^{\eta})\leq
2\sqrt{\epsilon}.
\end{equation}
Let us now explain (\ref{e6}) separately in the three cases.  Assume
that the above event holds.  If $k=1$, any loop from
$\omega_{k,\epsilon,R}^{\eta}$ or $\omega_{k,\eta,R}^{\eta}$ is
either inside or outside $\mathcal {C}$, and the loop configuration
outside $\mathcal {C}$ is identical for
$\omega_{k,\epsilon,R}^{\eta}$ and $\omega_{k,\eta,R}^{\eta}$. Then
(\ref{e6}) holds obviously.  If $k=2$, the loops entirely outside
$\mathcal {C}$ are identical for $\omega_{k,\epsilon,R}^{\eta}$ and
$\omega_{k,\eta,R}^{\eta}$. Furthermore, both
$\omega_{k,\epsilon,R}^{\eta}$ and $\omega_{k,\eta,R}^{\eta}$ have a
unique loop crossing the 2-circuit $\mathcal {C}$ which is composed
of two curves, one outside $\mathcal {C}$ and the other inside, and
the outside one is identical for $\omega_{k,\epsilon,R}^{\eta}$ and
$\omega_{k,\eta,R}^{\eta}$.  From this one gets (\ref{e6}) easily.
Suppose $k=4$, the loops entirely outside $\mathcal {C}$ are
identical for $\omega_{k,\epsilon,R}^{\eta}$ and
$\omega_{k,\eta,R}^{\eta}$.  Furthermore, if $U_{\mathcal {C}}=0$,
both $\omega_{k,\epsilon,R}^{\eta}$ and $\omega_{k,\eta,R}^{\eta}$
have a unique loop crossing the 4-circuit $\mathcal {C}$ which is
composed of four curves, two outside $\mathcal {C}$ and the others
inside, and the outside ones are identical for
$\omega_{k,\epsilon,R}^{\eta}$ and $\omega_{k,\eta,R}^{\eta}$; if
$U_{\mathcal {C}}=1$, both $\omega_{k,\epsilon,R}^{\eta}$ and
$\omega_{k,\eta,R}^{\eta}$ have exactly two loops crossing $\mathcal
{C}$, with each loop composed of two curves, one outside $\mathcal
{C}$ and the other inside, and the outside ones are identical for
$\omega_{k,\epsilon,R}^{\eta}$ and $\omega_{k,\eta,R}^{\eta}$.  Then
one obtains (\ref{e6}).

Combining (\ref{e5}) and (\ref{e6}), for each $\delta>0,\epsilon>0$,
there exists $\eta_0(\delta,\epsilon)>0$ such that for each
$\eta<\eta_0$, we can couple $\mu_R[\cdot|\mathcal
{A}_k(\epsilon,R)]$ and $\mu_R^\eta[\cdot|\mathcal
{A}_k^{\eta}(\eta,R)]$ such that with probability
$1-\delta-\epsilon^{\alpha/2}$,
\begin{equation}\label{e8}
\dist(\omega_{k,\epsilon,R},\omega_{k,\eta,R}^{\eta})\leq
\delta+2\sqrt{\epsilon}.
\end{equation}
Lemma \ref{l6} says that there exist subsequential limits of
$\mu_R^\eta[\cdot|\mathcal {A}_k^{\eta}(\eta,R)]$ as
$\eta\rightarrow 0$, (\ref{e8}) implies the uniqueness of the limit,
and we denote it by $\mu_{k,R}$.  (\ref{e8}) also implies that
$\mu_R[\cdot|\mathcal {A}_k(\epsilon,R)]$ converges in law to
$\mu_{k,R}$ as $\epsilon\rightarrow 0$.
\end{proof}

The conformal invariance of CLE$_6$ is expressed in the following
theorem, which will be used in the proof of Theorem \ref{t6}.

\begin{theorem}[\cite{32}]\label{t3}
Let $D,D'$ be two Jordan domains and let $f:\overline{D}\rightarrow
\overline{D'}$ be a continuous function that maps $D$ conformally
onto $D'$.  Then the CLE$_6$ in $\overline{D'}$ is distributed like
the image under $f$ of the CLE$_6$ in $\overline{D}$.
\end{theorem}

To prove Theorem \ref{t6}, we also need the following lemma about
conformal transformations, which is Corollary 3.25 in \cite{17}.

\begin{lemma}[\cite{17}]\label{l1}
Let $D,D'$ be two Jordan domains.  If $f: D\rightarrow D'$ is a
conformal transformation with $z\in D$, then for all $0<r<1$ and all
$|w-z|\leq r \dist(z,\partial D)$,
\begin{equation*}
|f(w)-f(z)|\leq \frac{4|w-z|}{(1-r)^2}\frac{\dist(f(z),\partial
D')}{\dist(z,\partial D)}.
\end{equation*}
\end{lemma}

\begin{proof}[Proof of Theorem \ref{t6}]
The first part of Theorem \ref{t6} is a generalization of Lemma
\ref{l4}. Its proof is basically the same as for Lemma \ref{l4}, and
we omit it.  Now we show the second part.  For any small
$\epsilon>0$, by the definitions of $f$ and $z'$, it is easy to see
that $f(\mathbb{D}(z;\epsilon))$ is a Jordan domain, $z'\in
f(\mathbb{D}(z;\epsilon))$ and
$\overline{f(\mathbb{D}(z;\epsilon))}\subset D'$. Further, Lemma
\ref{l1} implies that the diameter of $f(\mathbb{D}(z;\epsilon))$
converges to zero as $\epsilon\rightarrow 0$.  By the definitions of
arm events and Theorem \ref{t3}, the image of $\mu_D[\cdot|\mathcal
{A}_k(\overline{D}\backslash \mathbb{D}(z;\epsilon))]$ under $f$ has
the same law as $\mu_{D'}[\cdot|\mathcal
{A}_k(\overline{D'}\backslash f(\mathbb{D}(z;\epsilon)))]$.  Then
the first part of Theorem \ref{t6} implies the second part of
Theorem \ref{t6} immediately.
\end{proof}

\begin{proof}[Proof of Theorem \ref{t4}]
By Proposition \ref{l3}, given any bounded domain $D$, for each
$\epsilon>0$, there exists a $R_0(D,\epsilon)>0$, such that for any
$R_2>R_1>R_0$ and any small enough $\eta$,  we can couple
$\mu_{R_1}^{\eta}[\cdot|\mathcal {A}_k^{\eta}(\eta,R_1)]$ and
$\mu_{R_2}^{\eta}[\cdot|\mathcal {A}_k^{\eta}(\eta,R_2)]$ such that
with probability at least $1-\epsilon$, the cluster boundaries or
portions of boundaries contained in $\overline{D}$ are identical.
Therefore, letting $\eta\rightarrow 0$ and using Lemma \ref{l4},
there is a coupling between $\mu_{k,R_1}$ and $\mu_{k,R_2}$, such
that with probability at least $1-\epsilon$ the loops or portions of
loops contained in $\overline{D}$ are identical.  Taking
$\epsilon\rightarrow 0$ and $R=R(\epsilon)\rightarrow \infty$, we
get that $\hat{I}_D\mu_{k,R}$ converges in law to a probability
measure.  For $D=\mathbb{D}_r$, we denote the above limiting measure
by $\mu_{k,r}'$. The above argument also implies that $\mu_{k,r}'$
on $(\Omega_r, \mathcal {B}_r)$, for $r>0$, satisfy the consistency
$\mu_{k,r_1}'=\hat{I}_{\mathbb{D}_{r_1}}\mu_{k,r_2}'$ conditions for
all $0<r_1<r_2$.  Then using Kolmogorov's extension theorem (see,
e.g., \cite{18}) we conclude that there exists a unique probability
measure $\mu_k$ on $(\Omega,\mathcal {B})$ with
$\mu_{k,r}'=\hat{I}_{\mathbb{D}_{r}}\mu_k$ for all $r>0$.  For any
domain $D\subset \mathbb{D}_r$, the above discussion implies that as
$R\rightarrow\infty$, $\hat{I}_D\mu_{k,R}\rightarrow
\hat{I}_D\mu_{k,r}'=\hat{I}_D\mu_k$.

By Lemma \ref{l6}, we let $\{\eta_j\}$ be a convergent subsequence
for $\mu_k^{\eta}$ and let $\mu_k'$ be the limit in distribution of
$\mu_k^{\eta_j}$ as $\eta_j\rightarrow 0$.  Now we show
$\mu_k'=\mu_k$.  To achieve this, it is enough to prove that
$\hat{I}_{\mathbb{D}_r}\mu_k'=\hat{I}_{\mathbb{D}_r}\mu_k$ for all
$r>0$, which is achieved as follows.

By the definition of $\mu_k^{\eta}$, for each $\epsilon>0$, there
exist $\eta_0>0,R_0>0$, such that for all $\eta<\eta_0$ and all
$R>R_0$,  we can couple $\hat{I}_{\mathbb{D}_r}\mu_k^{\eta}$ and
$\hat{I}_{\mathbb{D}_r}\mu_R^{\eta}[\cdot|\mathcal
{A}_k^{\eta}(\eta,R)]$ such that with probability at least
$1-\epsilon$,
\begin{equation*}
\textrm{Dist}(\omega_{k,r}^{\eta},\omega_{k,r,R}^{\eta})\leq
\epsilon,
\end{equation*}
where $\omega_{k,r}^{\eta},\omega_{k,r,R}^{\eta}$ are the
configurations under these two laws.  Using Lemma \ref{l4} and the
definition of $\mu_k'$, by taking $\eta_j\rightarrow 0$, we can
couple $\hat{I}_{\mathbb{D}_r}\mu_k'$ and
$\hat{I}_{\mathbb{D}_r}\mu_{k,R}$ such that with probability at
least $1-\epsilon$,
\begin{equation*}
\textrm{Dist}(\omega_{k,r}',\omega_{k,r,R})\leq \epsilon,
\end{equation*}
where $\omega_{k,r}',\omega_{k,r,R}$ are the configurations under
these two laws.  Taking $R\rightarrow \infty$ and then
$\epsilon\rightarrow 0$, by the first part of the proof of Theorem
\ref{t4}, we have
$\hat{I}_{\mathbb{D}_r}\mu_k'=\hat{I}_{\mathbb{D}_r}\mu_k$.
\end{proof}

\section{Winding numbers}\label{s4}

In this section we will prove our main results concerning the
variance estimate and CLT for winding numbers of the arms in the
2-arm case.  We will use two-sided radial SLE$_6$, which is
introduced below.  We assume that the reader is familiar with the
basic theory of SLE. (See, for instance, Lawler's book \cite{17}.)
For the basic results regarding two-sided radial SLE, we refer to
\cite{48,49,51,53}.

\subsection{Winding for two-sided SLE}\label{s41}

To introduce two-sided radial SLE,  we need the notion of Green's
function for chordal SLE.  Roughly speaking, the Green's function
gives the normalized probability that the chordal SLE path goes
through an interior point. Before stating the precise definition, we
set up some notation.  If $D$ is a simply connected domain with
$z\in D$, we let $\Upsilon_D(z)$ be twice the conformal radius of
$z$ in $D$; that is, if $f:\mathbb{D}\rightarrow D$ is a conformal
transformation with $f(0)=z$, then $\Upsilon_D(z)=2|f'(0)|$. Suppose
$0<\kappa<8$, $a,b\in \partial D$, let $\gamma=\gamma_{D,a,b}$
denote chordal SLE$_{\kappa}$ path from $a$ to $b$ in
$\overline{D}$.  Let $D_{\infty}$ denote the component of
$D\backslash \gamma$ containing $z$.  The \emph{Green's function}
$G_D(z;a,b)$ for $\gamma$ is defined by
\begin{equation*}
\lim_{\epsilon\rightarrow
0}\epsilon^{d-2}P[\Upsilon_{D_{\infty}}(z)<\epsilon]=C_*G_D(z;a,b),
\end{equation*}
where $d:=1+\kappa/8$ is the Hausdorff dimension of SLE$_{\kappa}$
path, $C_*:=2[\int_0^{\pi}\sin^{8/\kappa}x d x]^{-1}$.  See, e.g.,
\cite{53} and Proposition 2.2 in \cite{51}.   In fact, for the
Euclidean distance, there also exists a constant $\hat{C}>0$ (the
value of $\hat{C}$ is unknown) such that
\begin{equation*}
\lim_{\epsilon\rightarrow
0}\epsilon^{d-2}P[\dist(z,\gamma)<\epsilon]=\hat{C}G_D(z;a,b).
\end{equation*}
Furthermore, Lawler and Rezaei proved that the Green's function
satisfies the following proposition (Theorem 2.3 in \cite{51}, see
also Theorem 2.3 in \cite{52}):
\begin{proposition}[\cite{51}]\label{p1}
Suppose $0<\kappa<8$. There exist $0<\hat{C},C,u<\infty$ (depending
on $\kappa$) such that the following holds.  Suppose $D$ is a simply
connected domain, $z\in D,a,b,\in\partial D$ and $\gamma$ is a
chordal SLE$_{\kappa}$ path from $a$ to $b$ in $\overline{D}$. Then,
for all $0<\epsilon<\dist(z,\partial D)/10$,
\begin{align*}
\left|\frac{P[\dist(z,\gamma)\leq\epsilon]}{\epsilon^{2-d}G_D(z;a,b)}-\hat{C}\right|\leq
C\left(\frac{\epsilon}{\dist(z,\partial D)}\right)^u.
\end{align*}
\end{proposition}

Assume $0<\kappa<8$ and $0<\alpha<2\pi$.  Roughly speaking, a
two-sided radial SLE$_{\kappa}$ path from 1 to $e^{i\alpha}$ through
0 in $\overline{\mathbb{D}}$ can be thought of as a chordal
SLE$_{\kappa}$ path $\gamma$ from $1$ to $e^{i\alpha}$ in
$\overline{\mathbb{D}}$, conditioned to pass through $0$ (see
Proposition \ref{p2} below).  The curve can be defined by weighting
$\gamma$ in the sense of the Girsanov theorem by Green's function in
the slit domain at 0.  More precisely, we parametrize $\gamma$ by
the radial parametrization (i.e., $g_t'(0)=e^t$), and let
$M_t:=G_{\mathbb{D}\backslash
\gamma[0,t]}(0;\gamma(t),e^{i\alpha})$, which is a local martingale.
Then using Girsanov's theorem, we can define a new probability
measure $P^*$ which corresponds to paths ``weighted locally by
$M_t$".  That is,
\begin{equation*}
P^*[V]=M_0^{-1}E[M_t1_V]~~\mbox{for $V\in\mathcal {F}_t$},
\end{equation*}
where $E$ denotes expectation with respect to $P$, $\mathcal {F}_t$
denotes the $\sigma$-algebra generated by $\{\hat{W}_s,0\leq s\leq
t\}$, and $\hat{W}$ is a standard Brownian motion and is the driving
function of $\gamma$.

Explicitly, if $\gamma$ denotes the two-sided radial SLE$_{\kappa}$
path from 1 to $e^{i\alpha}$ through 0 in $\overline{\mathbb{D}}$
stopped when it reaches $0$, $D_t$ denotes the connected component
of $\mathbb{D}\backslash\gamma(0,t]$ containing the origin, and
$g_t$ (two-sided radial SLE$_{\kappa}$): $D_t\rightarrow \mathbb{D}$
is the conformal transformation with $g_t(0)=0,g_t'(0)=e^t$, then
$g_t$ can be obtained from solving the initial value problem
\begin{align}
&\partial_tg_t(z)=g_t(z)\frac{e^{iU_t}+g_t(z)}{e^{iU_t}-g_t(z)},\label{e17}\\
&d\Theta_t=2\cot\left[\frac{\Theta_t}{2}\right]dt+\sqrt{\kappa}dW_t,~~\Theta_0=\alpha,~~-dU_t=\cot\left[\frac{\Theta_t}{2}\right]dt+\sqrt{\kappa}dW_t,\label{e27}
\end{align}
where $W$ is a standard Brownian motion with respect to $P^*$.
Further, if we write $g_t(e^{i\alpha})=e^{iV_t}$, then
\begin{equation}\label{e16}
\Theta_t=V_t-U_t.
\end{equation}
Note that we write the equation slightly differently than in
\cite{48,49}, where the authors added a parameter $a=2/\kappa$ that
gives a linear time change, and wrote $2U_t$ in the exponent in
(\ref{e17}).

Given $r>0$ and a curve $\gamma$, let $\tau_r=\tau_r(\gamma)$ be the
first hitting time with $\partial \mathbb{D}_r$ of $\gamma$.  The
following proposition is the analog of Proposition 2.13 in
\cite{53}, replacing conformal radius with Euclidean distance.  The
justification for calling two-sided radial SLE ``chordal SLE
conditioned to go through an interior point" comes from this
proposition.
\begin{proposition}\label{p2}
Let $0<\kappa<8$ and $0<\alpha<2\pi$.  There exist $0<u,C<\infty$
(depending on $\kappa$) such that the following is true.  Suppose
$\gamma$ is a chordal SLE$_{\kappa}$ path from $1$ to $e^{i\alpha}$.
Suppose $0<\epsilon<1/10,0<\epsilon'<\epsilon/10$.  Let $\mu',\mu^*$
be the two probability measures on $\{\gamma(t):0\leq
t\leq\tau_{\epsilon}\}$ corresponding to chordal SLE$_{\kappa}$
conditioned on the event $\{\tau_{\epsilon'}<\infty\}$ and two-sided
radial SLE$_{\kappa}$ through $0$, respectively.  Then $\mu',\mu^*$
are mutually absolutely continuous with respect to each other and
the Radon-Nikodym derivative satisfies
\begin{align*}
\left|\frac{d\mu^*}{d\mu'}-1\right|\leq
C\left(\frac{\epsilon'}{\epsilon}\right)^u.
\end{align*}
\end{proposition}
\begin{proof}
Let $P[\cdot]$ denote the law of the entire $\gamma$ and let
$P_{\epsilon}[\cdot]$ denote the law of $\{\gamma(t):0\leq
t\leq\tau_{\epsilon}\}$ restricted to the event
$\{\tau_{\epsilon}<\infty\}$.  From the definitions of $\mu',\mu^*$,
we know that
\begin{equation*}
d\mu^*=\frac{M_{\tau_{\epsilon}}}{M_0}dP_{\epsilon},~~d\mu'=\frac{P[\dist(0,\gamma)\leq\epsilon'|\gamma[0,\tau_{\epsilon}]]}{P[\dist(0,\gamma)\leq\epsilon']}dP_{\epsilon}.
\end{equation*}
So $\mu'$ and $\mu^*$ are mutually absolutely continuous.  Denote by
$E$ the expectation with respect to $P$, by $\mathcal
{F}_{\epsilon}$ the $\sigma$-algebra generated by
$\gamma[0,\tau_{\epsilon}]$, by $T$ the time that $\gamma$ reaches
$e^{i\alpha}$, by $P_{\mathbb{D}\backslash
\gamma[0,\tau_{\epsilon}]}$ the law of $\gamma[\tau_{\epsilon},T]$.
Using Proposition \ref{p1}, we have that for each $V\in\mathcal
{F}_{\epsilon}$,
\begin{align*}
&\mu^*(V)=M_0^{-1}E[M_{\tau_{\epsilon}}1_V]\\
&=G_{\mathbb{D}}(0;1,e^{i\alpha})^{-1}E[G_{\mathbb{D}\backslash
\gamma[0,\tau_{\epsilon}]}(0;\gamma(\tau_{\epsilon}),e^{i\alpha})1_V]\\
&=G_{\mathbb{D}}(0;1,e^{i\alpha})^{-1}E[E[G_{\mathbb{D}\backslash
\gamma[0,\tau_{\epsilon}]}(0;\gamma(\tau_{\epsilon}),e^{i\alpha})|\mathcal
{F}_{\epsilon}]1_V]\\
&=(\epsilon')^{2-d}\hat{C}(1+O((\epsilon')^u))P[\dist(0,\gamma)\leq\epsilon']^{-1}\times\\
&~~~~~E[E[(\epsilon')^{d-2}\hat{C}^{-1}(1+O((\epsilon'/\epsilon)^u))P_{\mathbb{D}\backslash
\gamma[0,\tau_{\epsilon}]}[\dist(0,\gamma[\tau_{\epsilon},T])\leq\epsilon']|\mathcal
{F}_{\epsilon}]1_V]\\
&=(1+O((\epsilon'/\epsilon)^u))P[\dist(0,\gamma)\leq\epsilon']^{-1}E[E[P_{\mathbb{D}\backslash
\gamma[0,\tau_{\epsilon}]}[\dist(0,\gamma[\tau_{\epsilon},T])\leq\epsilon']|\mathcal
{F}_{\epsilon}]1_V]\\
&=(1+O((\epsilon'/\epsilon)^u))\mu'(V).
\end{align*}
Then the result follows from the above inequality.
\end{proof}

The following lemma for two-sided radial SLE is an analog of Theorem
7.2 for radial SLE in \cite{47}.

\begin{lemma}\label{l12}
Let $0<\kappa<8$.  Suppose $\gamma$ is a two-sided radial
SLE$_\kappa$ path from $1$ to $-1$ through $0$ in
$\overline{\mathbb{D}}$ stopped when it reaches $0$.  Let $T\geq 0$,
and $\theta_{\kappa}(T)$ be the winding number of the path
$\gamma[0,T]$ around $0$.  Then there exist constants $C_0,C_1>0$
depending only on $\kappa$, such that for all $s>0$,
\begin{equation}\label{e21}
P^*\left[\big|T+\log|\gamma(T)|\big|>s\right]\leq C_0\exp(-C_1s),
\end{equation}
and
\begin{equation}\label{e22}
P^*\left[\left|\theta_{\kappa}(T)+\frac{\sqrt{\kappa}}{2}W_T\right|>s\right]\leq
C_0\exp(-C_1s).
\end{equation}
\end{lemma}

\begin{proof}
Schwarz Lemma and the Koebe $1/4$ Theorem give
\begin{equation}\label{e23}
\dist(0,\gamma[0,T])\leq e^{-T}=1/g'_t(0)\leq
4\dist(0,\gamma[0,T])\leq4|\gamma(T)|,
\end{equation}
which implies
\begin{equation}\label{e25}
\log|\gamma(T)|\geq -T+\log4.
\end{equation}
By Theorem 3 in \cite{49}, there exist $C_2(\kappa),C_3(\kappa)>0$
such that for all $k,n\in \mathbb{N}$,
\begin{equation}\label{e24}
P^*[\gamma[\tau_{e^{-n-k}},\infty)\cap
\partial\mathbb{D}_{e^{-k}}\neq\emptyset]\leq C_2\exp(-C_3n).
\end{equation}
(\ref{e23}) and (\ref{e24}) imply that there exist $C_4,C_5>0$, such
that
\begin{equation*}
P^*[T+\log|\gamma(T)|>s]\leq P^*[\gamma[\tau_{e^{-T}},\infty)\cap
\partial\mathbb{D}_{e^{-T+s}}\neq\emptyset]\leq C_4\exp(-C_5s).
\end{equation*}
Combining this with (\ref{e25}), we get (\ref{e21}).

The proof of (\ref{e22}) is similar to that of (7.3) in \cite{47},
we just sketch it here.  For $t\in[0,T]$, let
$y(t):=\arg[g_t(\gamma(T))],$ where $\arg$ is chosen to be
continuous in $t$.  Using the argument in the proof of (7.3) in
\cite{47}, one can show that
\begin{equation}\label{e26}
\theta_{\kappa}(T)=U_T-U_0+y(0)-y(T).
\end{equation}
By (\ref{e27}), we have
\begin{equation}\label{e28}
-U_t=\frac{\Theta_t}{2}+\frac{\sqrt{\kappa}}{2}W_t.
\end{equation}
From (\ref{e16}), we have $0\leq\Theta_t\leq 2\pi$ for all $t\geq
0$.  Then, by (\ref{e26}) and (\ref{e28}), proving (\ref{e22}) boils
down to prove the appropriate bound on the tail of $|y(0)-y(T)|$.
Let $\tau_1$ be the largest $t\in[0,T]$ such that
$\log|g_t(\gamma(T))|\leq -1$, and set $\tau_1=0$ if such a $t$ does
not exist.  Analogous to the proof of (7.7) in \cite{47}, it can be
shown that $|y(0)-y(\tau_1)|<\infty$. Now let us bound
$|y(\tau_1)-y(T)|$.  Set $t_0=T$, and inductively, let $t_j$ be the
last $t\in[0,t_{j-1}]$ such that
$\pi/2=\min\{|\sqrt{\kappa}U_t-\sqrt{\kappa}U_{t_{j-1}}-2\pi
n|:n\in\mathbb{Z}\}$, and set $t_j=0$ if no such $t$ exists.
Analogous to the proof of (7.8) in \cite{47}, one can show that for
every $a>0$ and $n\in\mathbb{N}$,
\begin{equation*}
P^*[|y(\tau_1)-y(T)|\geq 2\pi n]\leq P^*[T-\tau_1\geq a]+P^*[t_n\geq
T-a].
\end{equation*}
Using (\ref{e21}), (\ref{e28}) and the argument at the end of the
proof of Theorem 7.2 in \cite{47}, choosing $a$ to be $n$ times a
very small constant, one can bound the two summands on right hand
side appropriately.
\end{proof}

\begin{remark}
Theorem 3 in \cite{49} is a result only for two-sided radial
SLE$_{\kappa}$ from $1$ to $-1$ through $0$.  Adapting the proof of
this result, one can get the analog for general two-sided radial
SLE$_{\kappa}$ from $1$ to $e^{i\alpha}$ through $0$ in
$\overline{\mathbb{D}}$, where $0<\alpha<2\pi$.  Using this,
following the proof of Lemma \ref{l12}, one can obtain the analog of
Lemma \ref{l12} for general two-sided radial SLE$_{\kappa}$.  For
the general case, it is expected that the corresponding $C_0$ and
$C_1$ depend only on $\kappa$, not on $\alpha$.  Combining Theorem
1.3 in \cite{48} and our proof of Lemma \ref{l12}, one can show this
for $0<\kappa\leq 4$.
\end{remark}

The following result gives exact second moment estimate for the
winding number of the two-sided radial SLE, which will be used to
give estimate for the winding number variance of the arms crossing a
long annulus in the 2-arm case.
\begin{lemma}\label{l13}
Let $E^*$ denote the expectation with respect to $P^*$.  $\gamma$
and $\theta_{\kappa}$ are as defined in Lemma \ref{l12}.  We have
\begin{equation*}
E^*\left[\theta_{\kappa}(\tau_{\epsilon})^2\right]=\left(\frac{\kappa}{4}+o(1)\right)\log\left(\frac{1}{\epsilon}\right)\mbox{
as }\epsilon\rightarrow 0.
\end{equation*}
\end{lemma}

\begin{proof}
Schwarz Lemma and the Koebe $1/4$ Theorem give
\begin{equation}\label{e31}
\epsilon\leq e^{-\tau_{\epsilon}}=1/g'_{\tau_{\epsilon}}(0)\leq
4\epsilon.
\end{equation}
Using this, similarly to the proof of (\ref{e22}), one can show that
there exist constants $C_0,C_1>0$ depending only on $\kappa$, such
that for all $s>0$,
\begin{equation}\label{e32}
P^*\left[\left|\theta_{\kappa}(\tau_{\epsilon})+\frac{\sqrt{\kappa}}{2}W_{\tau_{\epsilon}}\right|>s\right]\leq
C_0\exp(-C_1s).
\end{equation}
Combining (\ref{e31}) and (\ref{e32}), one obtains Lemma \ref{l13}
easily.
\end{proof}

\subsection{Convergence of discrete exploration to
SLE$_6$}\label{s42}

Assume $0<r<1$.  Similarly to the definition of $\mathcal
{A}_r^{\eta}$ defined above Proposition \ref{p5}, for chordal
SLE$_6$ path $\gamma_{\mathbb{D},1,-1}$ we define event
\begin{equation*}
\mathcal {A}_{r}:=\{\gamma_{\mathbb{D},1,-1}\cap
\partial \mathbb{D}_r\neq\emptyset\}.
\end{equation*}

The following lemma is a corollary of the well-known result that the
percolation exploration path converges in the scaling limit to the
chordal SLE$_6$ path (see, e.g., Theorem 5 in \cite{40}).  The proof
is standard and easy.

\begin{lemma}\label{l10}
Let $0<r'\leq r<1$.
$\gamma^{\eta}_{\mathbb{D},1,-1}[0,\tau_r^{\eta}]$ conditioned on
$\mathcal {A}_{r'}^{\eta}$ converges in distribution to stopped
chordal SLE$_6$ path $\gamma_{\mathbb{D},1,-1}[0,\tau_r]$
conditioned on $\mathcal {A}_{r'}$ with respect to the uniform
metric (\ref{e1}) as $\eta\rightarrow 0$.
\end{lemma}

\begin{proof}
Let $P^{\eta}$ and $P$ denote the laws of
$\gamma^{\eta}_{\mathbb{D},1,-1}$ and $\gamma_{\mathbb{D},1,-1}$,
respectively.  We claim that for each $0<r<1$, there exists a
constant $C>0$ (depending on $r$), such that
\begin{equation}\label{e71}
\lim_{\eta\rightarrow 0}P^{\eta}[\mathcal {A}_{r}^{\eta}]=P[\mathcal
{A}_r]>C.
\end{equation}
Moreover, in any coupling of $\{P^{\eta}\}$ and $P$ on
$(\Omega,\mathcal {F})$ in which
$\textrm{d}(\gamma^{\eta}_{\mathbb{D},1,-1},\gamma_{\mathbb{D},1,-1})\rightarrow
0$ a.s. as $\eta\rightarrow 0$, we have
\begin{equation}\label{e70}
\hat{P}[\{\gamma^{\eta}_{\mathbb{D},1,-1}\in \mathcal
{A}_r^{\eta}\}\Delta\{\gamma_{\mathbb{D},1,-1}\in \mathcal
{A}_r\}]\rightarrow 0\mbox{ as }\eta\rightarrow 0,
\end{equation}
where $\hat{P}[\cdot]$ denotes the coupling measure.  The proof of
the claim is analogous to that of Lemma \ref{l11}: By Theorem 5 in
\cite{40}, we can couple $\{P^{\eta}\}$ and $P$ on $(\Omega,\mathcal
{F})$ such that
$\textrm{d}(\gamma^{\eta}_{\mathbb{D},1,-1},\gamma_{\mathbb{D},1,-1})\rightarrow
0$ a.s. as $\eta\rightarrow 0$.  Now let us show (\ref{e70}).  For
each small $\epsilon>0$ and $\eta<\epsilon$,
\begin{align*}
\hat{P}[&\{\gamma_{\mathbb{D},1,-1}\in \mathcal
{A}_r\}\backslash\{\gamma^{\eta}_{\mathbb{D},1,-1}\in \mathcal
{A}_r^{\eta}\}]\\
&\leq\hat{P}[\textrm{d}(\gamma^{\eta}_{\mathbb{D},1,-1},\gamma_{\mathbb{D},1,-1})\geq\epsilon]+\hat{P}[\gamma^{\eta}_{\mathbb{D},1,-1}\cap
\partial \mathbb{D}_{r+\epsilon}\neq\emptyset,\gamma^{\eta}_{\mathbb{D},1,-1}\cap
\partial \mathbb{D}_r^{\eta}=\emptyset].
\end{align*}
The event in the second term implies a half-plane 3-arm event from
the $2\epsilon$-neighborhood of $\partial\mathbb{D}_r$ to a distance
of unit order, whose probability goes to zero as
$\epsilon\rightarrow 0$.  Then we get that
$\hat{P}[\{\gamma^{\eta}_{\mathbb{D},1,-1}\in \mathcal
{A}_r^{\eta}\}\Delta\{\gamma_{\mathbb{D},1,-1}\in \mathcal
{A}_r\}]\rightarrow 0$ as $\eta\rightarrow 0$.  The other direction
is easy to prove and the details are omitted.  Then we get
(\ref{e70}). RSW, FKG and (\ref{e70}) imply (\ref{e71}) immediately.

Let $0<r'\leq r<1$.  Conditioned on $\mathcal {A}_{r'}^{\eta}$ and
$\mathcal {A}_{r'}$, let $\gamma_r^{\eta}[0,1]$ and $\gamma_r[0,1]$
be the respective reparametrized curve of
$\gamma^{\eta}_{\mathbb{D},1,-1}[0,\tau_r^{\eta}]$ and
$\gamma_{\mathbb{D},1,-1}[0,\tau_r]$.  Notice that
$\{\gamma_r^{\eta}\}$ satisfies the conditions in \cite{33} and thus
has a scaling limit in terms of continuous curves along subsequence
of $\eta$.  We claim that for every subsequence limit
$\widetilde{\gamma}_r[0,1]$, $\widetilde{\gamma}_r[0,1)\subset
\overline{\mathbb{D}}\backslash\mathbb{D}_r$ almost surely.  Then
the fact that $\gamma_r^{\eta}$ converges in distribution to
$\gamma_r$ easily follows from our two claims and Theorem 5 in
\cite{40}.  It remains to show this claim.  Assume that this is not
the case for the limit $\widetilde{\gamma}_r$ along some subsequence
$\{\eta_k\}_{k\in\mathbb{N}}$.  Then with positive probability
$\widetilde{\gamma}_r[0,1)\nsubseteq
\overline{\mathbb{D}}\backslash\mathbb{D}_r$.  Suppose this happens.
We can find coupled versions of $\gamma_r^{\eta_k}$ and
$\widetilde{\gamma}$ on $(\Omega,\mathcal {B})$ such that
$\textrm{d}(\gamma_r^{\eta_k},\widetilde{\gamma}_r)\rightarrow 0$
a.s. as $k\rightarrow\infty$.  Using this coupling, for each small
$\epsilon>0$ and $\eta_k<\epsilon/10$, we have a half-plane 3-arm
event produced by $\gamma_r^{\eta_k}$ from the
$\epsilon$-neighborhood of $\widetilde{\gamma}_r(1)$ to a distance
of unit order.  As $\eta_k\rightarrow 0$, we can let
$\epsilon\rightarrow 0$, in which case the probability of the seeing
this event goes to zero, leading to a contradiction.
\end{proof}

In order to derive winding number estimates for the arms from the
corresponding result for two-sided radial SLE$_6$, we need the
following lemma.
\begin{lemma}\label{l16}
Suppose $0<\epsilon<1/10$.  Let $\gamma^{\eta}$ and $\gamma$ denote
$\gamma^{\eta}_{\mathbb{D},1,-1}$ conditioned on $\mathcal
{A}^{\eta}$ and two-sided radial SLE$_6$ path from $1$ to $-1$
through $0$ in $\overline{\mathbb{D}}$, respectively.  Then
$\gamma^{\eta}[0,\tau_{\epsilon}^{\eta}]$ converges in distribution
to $\gamma[0,\tau_{\epsilon}]$ with respect to the uniform metric
(\ref{e1}), as $\eta\rightarrow 0$.
\end{lemma}

\begin{proof}
For $0<\epsilon'<1$, let $\gamma_{\epsilon'}$ denote
$\gamma_{\mathbb{D},1,-1}$ conditioned on $\mathcal
{A}_{\epsilon'}$, and let $\gamma_{\epsilon'}^{\eta}$ denote
$\gamma_{\mathbb{D},1,-1}^{\eta}$ conditioned on $\mathcal
{A}_{\epsilon'}^{\eta}$.  By Proposition \ref{p5}, for all
$\eta<\epsilon'/10<\epsilon/100$, we can couple $\gamma^{\eta}$ and
$\gamma_{\epsilon'}^{\eta}$, such that with probability at lest
$1-(\epsilon'/\epsilon)^{\beta}$,
\begin{equation}\label{e64}
\textrm{d}(\gamma^{\eta}[0,\tau_{\epsilon}^{\eta}],\gamma_{\epsilon'}^{\eta}[0,\tau_{\epsilon}^{\eta}])=0.
\end{equation}
By Lemma \ref{l10}, for each $0<\delta<1$, there exists
$\eta_0(\delta,\epsilon')$, such that for each $\eta<\eta_0$ and
$0<\epsilon'<\epsilon<1/10$, there is a coupling of
$\gamma_{\epsilon'}^{\eta}$ and $\gamma_{\epsilon'}$, such that with
probability at least $1-\delta$,
\begin{equation}\label{e65}
\textrm{d}(\gamma_{\epsilon'}^{\eta}[0,\tau_{\epsilon}^{\eta}],\gamma_{\epsilon'}[0,\tau_{\epsilon}])\leq
\delta.
\end{equation}
By Proposition \ref{p2}, for each $0<\delta<1$, there exists
$\epsilon_0'(\delta,\epsilon)$, such that for each
$0<\epsilon'<\epsilon_0'$ there is a coupling of
$\gamma_{\epsilon'}$ and $\gamma$, such that with probability at
least $1-\delta$,
\begin{equation}\label{e66}
\textrm{d}(\gamma_{\epsilon'}[0,\tau_{\epsilon}],\gamma[0,\tau_{\epsilon}])\leq
\delta.
\end{equation}
Combining (\ref{e64}), (\ref{e65}) and (\ref{e66}) gives the desired
result.
\end{proof}

\subsection{Moment bounds on the winding of discrete
exploration}\label{s43}

Define $R^{\eta}(r,R):=\{z\in\eta\mathbb{T}: |\arg (z)|<\pi/10\}\cap
A^{\eta}(r,R)$.  We say a path $\gamma\subset R^{\eta}(r,R)$ is a
\emph{crossing} of $R^{\eta}(r,R)$ if the endpoints of $\gamma$ lie
adjacent (Euclidean distance smaller than $\eta$) to the rays of
argument $\pm\frac{\pi}{10}$ respectively.  By Lemma 2.1 in
\cite{24}, we obtain the following lemma, which implies that it is
very unlikely that there is an arm with large winding in an annulus.

\begin{lemma}[\cite{24}]\label{l20}
There exist constants $C_1,C_2,K_0>0$, such that for all $K>K_0$ and
$\eta\leq r<R$,
$$P[\exists\mbox{ $\lfloor K\log(R/r)\rfloor$ disjoint blue crossings of }R^{\eta}(r,R)]\leq C_1\exp[-C_2K\log(R/r)].$$
\end{lemma}

The following three lemmas give moment bounds for the winding
numbers of percolation exploration path.  Let us define some
notation before stating the results.

Suppose $r<R$.  For a curve $\gamma$ hitting with $\partial
\mathbb{D}_R^{\eta}$ before hitting with $\partial
\mathbb{D}_r^{\eta}$, denote by $T_{R,r}^{\eta}$ the last hitting
time with $\partial \mathbb{D}_R^{\eta}$ of $\gamma$ before time
$\tau_r^{\eta}$.

Recall the definition of $P_R^*[\cdot]$ which is defined after the
definition of good faces.  Denote by $E_R^*$ the expectation with
respect to $P_R^*[\cdot]$.  Let $\Theta$ be the good faces around
$\partial \mathbb{D}_R^{\eta}$ under $P_R^*[\cdot]$.  Denote by
$\gamma_R^*$ the percolation exploration path connecting the
endpoints of $\Theta$ stopped when it reaches $\eta H_0$.

Unless specified otherwise, in the rest of this paper, we denote by
$E=E^{\eta}$ the expectation with respect to $P[\cdot|\mathcal
{A}^{\eta}]$, and by $\gamma=\gamma^{\eta}$ the percolation
exploration path $\gamma^{\eta}_{\mathbb{D},1,-1}$ conditioned on
$\mathcal {A}^{\eta}$.  For simplicity, we will omit the superscript
$\eta$ of $\gamma^{\eta},\tau_r^{\eta}$ and $T_{R,r}^{\eta}$ when it
is clear that we are talking about the the discrete percolation
model.
\begin{lemma}\label{l22}
Let $\eta\leq r<R\leq 1$. We have
\begin{equation*}
|E\theta(\gamma[0,\tau_r])|\leq\pi,~~|E\theta(\gamma[T_{R,r},\tau_r])|\leq
\pi,~~|E\theta(\gamma[\tau_R,\tau_r])|\leq 2\pi\mbox{ and
}E_R^*\theta(\gamma_R^*)=0.
\end{equation*}
\end{lemma}
\begin{proof}
First let us show the first inequality.  Conditioned on $\mathcal
{A}^{\eta}$, consider the time-reversal of
$\gamma_{\mathbb{D},1,-1}^{\eta}$ stopped when it reaches $\eta
H_0$, denoted by $\gamma'$.  By the symmetry of the lattice, it is
easy to see that
$E\theta(\gamma[0,\tau_r])=-E\theta(\gamma'[0,\tau_r])$. It is
obvious that
$|\theta(\gamma[0,\tau_r])-\theta(\gamma'[0,\tau_r])|\leq 2\pi$.
These two observations immediately imply
$|E\theta(\gamma[0,\tau_r])|\leq\pi$.  Similarly one can show the
second inequality.  Using the first inequality, we get the third
one:
\begin{equation*}
|E\theta(\gamma[\tau_R,\tau_r])|\leq
|E\theta(\gamma[0,\tau_r])-E\theta(\gamma[0,\tau_R])|\leq 2\pi.
\end{equation*}

Now let us show $E_R^*\theta(\gamma_R^*)=0$.  For any fixed good
faces $\Theta$ around $\partial \mathbb{D}_R^{\eta}$, denote by
$\Theta'$ the mirror image of $\Theta$ with opposite colors with
respect to the imaginary axis. It is obvious that $\Theta\neq
\Theta'$, $P_R^*[\Theta]=P_R^*[\Theta']$ and
$E_R^*[\theta(\gamma_R^*)|\Theta]=-E_R^*[\theta(\gamma_R^*)|\Theta']$.
Then $E_R^*\theta(\gamma_R^*)=0$ follows immediately.
\end{proof}

\begin{lemma}\label{l19}
There exist constants $C_1,C_2,C_3>0$, such that for all $\eta\leq
r\leq R/2\leq 1/2$,
\begin{align}
&E|\theta(\gamma[T_{R,r},\tau_{r}])|\leq
\sqrt{C_1\log(R/r)},\label{e59}\\
&E[\theta(\gamma[T_{R,r},\tau_{r}])^2]\leq
C_1\log(R/r),\label{e74}\\
&E[\theta(\gamma[T_{R,r},\tau_{r}])^4]\leq
C_2[\log(R/r)]^4,\label{e75}\\
&E[\theta(\gamma[\tau_R,T_{R,r}])^2]\leq C_3.\label{e58}
\end{align}
\end{lemma}
\begin{proof}
First let us show (\ref{e74}).  In \cite{24}, conditioned on
$\mathcal {A}_2^{\eta}(\eta,1)$, we showed that the winding number
variance of the arm connecting the two boundary pieces of
$A^{\eta}(\eta,1)$ is $O(\log(1/\eta))$ (Theorem 1.1 in \cite{24})
by a martingale method.  Conditioned on $\mathcal {A}^{\eta}$, one
can use the same method to show that $Var(\theta_{\eta})$ is again
$O(\log(1/\eta))$.  Furthermore, with a little modification for our
setting, one can also use this method to show that there exists a
constant $C_0>0$, such that for all $\eta\leq r\leq R/2\leq 1/2$,
\begin{equation}\label{e55}
Var|\theta(\gamma[T_{R,r},\tau_{r}])|\leq C_0\log(R/r).
\end{equation}
We left the details to the reader.  Lemma \ref{l22} says that
\begin{equation}\label{e56}
|E\theta(\gamma[T_{R,r},\tau_{r}])|\leq \pi.
\end{equation}
Then (\ref{e55}) and (\ref{e56}) imply (\ref{e74}).  (\ref{e74}) and
Cauchy-Schwarz inequality imply (\ref{e59}) immediately.

We show (\ref{e75}) now.  We claim that there exist $C_4,C_5,C_6>0$,
such that for all $\eta\leq r\leq R/2\leq 1/2$ and $x\geq
C_4\log(R/r)$,
\begin{equation*}
P[|\theta(\gamma[T_{R,r},\tau_r])|\geq x|\mathcal {A}^{\eta}]\leq
C_5\exp(-C_6x).
\end{equation*}
Then (\ref{e75}) follows from the claim immediately. The claim is
proved as follows.  Choosing $C_4$ large enough, we have
\begin{align*}
P[|\theta(\gamma[T_{R,r},\tau_r])|\geq x|\mathcal
{A}^{\eta}]&\leq\frac{C_7P[\mbox{$\exists\lfloor x/2\pi\rfloor-2$
disjoint blue crossings of $R^{\eta}(r,R)$}]}{P[\mathcal
{A}_2(r,R)]}\\
&~~~~~\mbox{ by quasi-multiplicativity and (\ref{e72})}\\
&\leq C_5\exp(-C_6x)\mbox{ by Lemma \ref{l20} and (\ref{e10})}.
\end{align*}

Now let us show (\ref{e58}).  Set
$N=\max\{\lceil\log_2(1/R)\rceil,\lceil\log_2(R/r)\rceil\}$. For
$0\leq j\leq N+1$, let
$R_j:=\min\{1,2^jR\},r_j:=\max\{r,(1/2)^jR\}$.  For $0\leq j\leq N$,
define event
$$\mathcal {B}_j:=\{\gamma[\tau_R,T_{R,r}]\cap(
\partial\mathbb{D}_{R_j}^{\eta}\cup\partial\mathbb{D}_{r_j}^{\eta})\neq\emptyset,\gamma[\tau_R,T_{R,r}]\subset A^{\eta}(r_{j+1},R_{j+1})\}.$$
There exist $C_8,C_9,C_{10}>0$, such that for all $\mathcal {B}_j$,
$0\leq j\leq N$,
\begin{align}
P[&\mathcal {B}_j|\mathcal {A}^{\eta}]\nonumber\\
&\leq \frac{C_8P[\mbox{$\exists$ bichromatic 3-arm crossing
$A^{\eta}(R,R_j)$ or $A^{\eta}(r_j,R)$},\mathcal
{A}_2^{\eta}(r_j,R_j)]}{P[\mathcal {A}_2^{\eta}(r_j,R_j)]}\nonumber\\
&~~~~~\mbox{ by
quasi-multiplicativity and (\ref{e72})}\nonumber\\
&\leq C_9\exp(-C_{10}j)\mbox{ by BK inequality and
(\ref{e10}).}\label{e62}
\end{align}

Moreover, using quasi-multiplicativity and a gluing argument with
FKG, RSW and Theorem 11 in \cite{19}, it is easy to show that there
exist $C_{11},C_{12}>0$ such that
\begin{equation}\label{e61}
P[\mathcal {B}_j,\mathcal {A}^{\eta}]\geq
C_{11}\exp(-C_{12}j)P[\mathcal {A}_2^{\eta}(R_{j+1},1)]P[\mathcal
{A}_2^{\eta}(\eta,r_{j+1})].
\end{equation}
We leave the details to the reader.  For simplicity, we let
$P[\mathcal {A}_2^{\eta}(x,x)]=1$ for any $x>0$ in the above
inequality and in the rest of the paper.

Hence, we can choose $C_{13},C_{14},C_{15}>0$ such that for all
$0\leq j\leq N$ and $x\geq C_{13}(j+1)$,
\begin{align}
P[|\theta(&\gamma[\tau_R,T_{R,r}])|\geq x|\mathcal {A}^{\eta},\mathcal {B}_j]\nonumber\\
&\leq\frac{P[\mbox{$\exists\lfloor x/(2\pi)\rfloor-2$ disjoint blue
crossings of }R^{\eta}(r_{j+1},R_{j+1})]}{C_{11}\exp(-C_{12}j)}\mbox{ by (\ref{e61})}\nonumber\\
&\leq C_{14}\exp(-C_{15}x)\mbox{ by Lemma \ref{l20}}.\label{e63}
\end{align}
Choosing $C_3$ large enough, (\ref{e58}) follows easily from
(\ref{e62}) and (\ref{e63}):
\begin{align*}
E\left[\theta(\gamma[\tau_R,T_{R,r}])^2\right]&\leq
\sum_{j=0}^{N}P[\mathcal {B}_j|\mathcal
{A}^{\eta}]E\left[\theta(\gamma[\tau_R,T_{R,r}])^2|\mathcal
{B}_j\right]\\
&\leq \sum_{j=0}^{N}C_{16}\exp(-C_{10}j)(j+1)^2\leq C_3.
\end{align*}
\end{proof}

The following lemma can be considered as a generalization of Lemma
\ref{l22}.

\begin{lemma}\label{l18}
There exists a constant $C>0$, such that for all $\eta\leq r<R/2$,
any given faces $\Theta$ around $\partial \mathbb{D}_R^{\eta}$, and
the percolation exploration path $\gamma$ connecting the endpoints
of $\Theta$ stopped when it reaches $\eta H_0$ conditioned on
$\mathcal {A}_{\Theta}^{\eta}(\eta,R)$, we have
\begin{equation*}
\left|E_{\Theta}[\theta(\gamma[0,\tau_r])]\right|\leq C,
\end{equation*}
where $E_{\Theta}$ is the expectation with respect to
$P[\cdot|\mathcal {A}_{\Theta}^{\eta}(\eta,R)]$.
\end{lemma}

\begin{proof}
For simplicity, we just show that
$\left|E_{\Theta}[\theta(\gamma)]\right|\leq C$, the proof of
$\left|E_{\Theta}[\theta(\gamma[0,\tau_r])]\right|\leq C$ is
essentially the same.  By Proposition \ref{p4}, there exist
$C_0,C_1>0$, such that for all $10\eta<R/2$, any fixed faces
$\Theta$ around $\partial \mathbb{D}_R^{\eta}$ and
$N:=\lfloor\log_2(R/\eta)\rfloor$, there is a coupling of
$P[\cdot|\mathcal {A}_{\Theta}^{\eta}(\eta,R)]$ and
$\{P_{(1/2)^jR}^*[\cdot]\}_{1\leq j\leq N}$, so that for all $1\leq
j\leq N $, with probability at least $1-\exp(-C_0j)$, the following
event $\mathcal {B}_j$ occurs: There exists $1\leq j^*\leq j$ such
that there exist good faces $\Theta_{j^*}$ around $\partial
\mathbb{D}^{\eta}_{(1/2)^{j^*}R}$ under $P[\cdot|\mathcal
{A}_{\Theta}^{\eta}(\eta,R)]$, and the configuration constraint in
$\overline{\Theta}_{j^*}$ under $P[\cdot|\mathcal
{A}_{\Theta}^{\eta}(\eta,R)]$ is the same as the configuration under
$P_{(1/2)^{j^*}R}^*[\cdot]$.  Furthermore, under this coupling for
all $1\leq j\leq N-1$, with probability at least $\exp(-C_1(j+1))$
the event $\mathcal {B}_j^c\mathcal {B}_{j+1}$ occurs.

Denote by $\hat{P}$ the coupling law, and by $\hat{E}$ the
expectation with respect to $\hat{P}$.  By Proposition \ref{p4} and
Lemma \ref{l22}, we have
\begin{align*}
\left|E_{\Theta}[\theta(\gamma)]\right|&=\left|\hat{E}[I_{\mathcal
{B}_1}\theta(\gamma)]+\hat{E}[I_{\mathcal
{B}_N^c}\theta(\gamma)]+\Sigma_{j=1 }^{N-1}\hat{E}[I_{\mathcal
{B}_j^c\mathcal
{B}_{j+1}}\theta(\gamma)]\right|\\
&\leq\hat{E}|\theta(\gamma[0,\tau_{R/2}])|+\exp(-C_0N)\hat{E}[|\theta(\gamma)||\mathcal
{B}_N^c]+\\
&~~~~~\Sigma_{j=1}^{N-1}\exp(-C_0j)\left|\hat{E}[\theta(\gamma[0,\tau_{(1/2)^{j+1}R}])|\mathcal
{B}_j^c\mathcal {B}_{j+1}]\right|.
\end{align*}
Then $\left|E_{\Theta}[\theta(\gamma)]\right|\leq C$ easily follows
from the following claim: There exists $C_2>0$, such that for all
$1\leq j\leq N-1$,
\begin{equation}\label{e50}
\hat{E}[|\theta(\gamma[0,\tau_{(1/2)^{j+1}R}])||\mathcal
{B}_j^c\mathcal {B}_{j+1}]\leq C_2j.
\end{equation}
Furthermore, there exist $C_3,C_4>0$ such that
\begin{equation}\label{e51}
\hat{E}|\theta(\gamma[0,\tau_{R/2}])|\leq C_3\mbox{ and
}\hat{E}|\theta(\gamma)|\mathcal {B}_N^c|\leq C_4N.
\end{equation}

Let us show (\ref{e50}) now.  By the coupling,  there is a constant
$C_1>0$ such that for all $1\leq j\leq N-1$,
\begin{equation}\label{e53}
\hat{P}[\mathcal {B}_j^c\mathcal {B}_{j+1}]\geq\exp(-C_1(j+1)).
\end{equation}
Without loss of generality, for the faces $\Gamma_1$ and $\Gamma_2$
of $\Theta=(\Gamma_1,\Gamma_2)$ (recall that we always assume that
$\Gamma_1$ is blue and $\Gamma_2$ is yellow), we assume that
$|\theta(\Gamma_1)|\leq |\theta(\Gamma_2)|$ (we think of the face as
a continuous curve by connecting the neighbor sites with line
segments).  By a gluing construction with RSW and FKG, it is easy to
show that
\begin{equation}\label{e73}
P[\mathcal {A}_{\Theta}^{\eta}(R/2,R)]\asymp P\left[\Gamma_1
\stackrel{\dot{\overline{\Theta}}}\leftrightarrow\partial
\mathbb{D}_{R/2}^{\eta}\right],
\end{equation}
where $\Gamma_1
\stackrel{\dot{\overline{\Theta}}}\leftrightarrow\partial
\mathbb{D}_r^{\eta}$ denotes that there exists a blue path
connecting $\Gamma_1$ and $\partial\mathbb{D}_r^{\eta}$ in the
interior of $\overline{\Theta}$ for $r<R$.  Then we know that there
exist $C_5,C_6,C_7>0$ such that for all $1\leq j\leq N-1$,
\begin{align}
P[\mathcal {A}_{\Theta}^{\eta}&((1/2)^jR,R)]\nonumber\\
&\geq C_5P[\mathcal {A}_2^{\eta}((1/2)^jR,R/2)]P[\mathcal
{A}_{\Theta}^{\eta}(R/2,R)]\mbox{ by quasi-multiplicativity}\nonumber\\
&\geq C_6\exp(-C_7j)P\left[\Gamma_1
\stackrel{\dot{\overline{\Theta}}}\leftrightarrow\partial
\mathbb{D}_{R/2}^{\eta}\right]\mbox{ by (\ref{e73}) and (\ref{e10})
}.\label{e54}
\end{align}

Conditioned on $\mathcal {A}_{\Theta}^{\eta}((1/2)^jR,R)$, we let
$\gamma_j$ be the b-path starting at an endpoint of $\Theta$ and
ending when it reaches $\partial \mathbb{D}_{(1/2)^{j}R}^{\eta}$
with yellow hexagons on its left. We can choose $C_8$ large enough,
such that the following inequalities hold:
\begin{align*}
\hat{P}[&|\theta(\gamma[0,\tau_{(1/2)^{j+1}R}])|\geq C_8j |\mathcal
{B}_j^c\mathcal
{B}_{j+1}]\\
&\leq\exp(C_1(j+1))P[|\theta(\gamma[0,\tau_{(1/2)^{j+1}R}])|\geq
C_8j|\mathcal {A}_{\Theta}^{\eta}(\eta,R)]\mbox{ by (\ref{e53})}\\
&\leq C_9\exp(C_1j)P[|\theta(\gamma_{j+1})|\geq C_8j|\mathcal
{A}_{\Theta}^{\eta}((1/2)^{j+1}R,R)]\mbox{ by
quasi-multiplicativity}\\
&\leq C_{10}\exp(C_{11}j)\frac{P[|\theta(\gamma_{j+1})|\geq
C_8j,\mathcal {A}_{\Theta}^{\eta}((1/2)^{j+1}R,R)]}{P\left[\Gamma_1
\stackrel{\dot{\overline{\Theta}}}\leftrightarrow\partial
\mathbb{D}_{R/2}^{\eta}\right]}\mbox{ by (\ref{e54})}.
\end{align*}
Observe that if $|\theta(\gamma_{j+1})|$ is very large, then
$\gamma_{j+1}$ will produce many crossings in the ``rectangle"
$R^{\eta}((1/2)^{j+1}R,2R)$, or $\gamma_{j+1}$ will cross
$A^{\eta}(R,2R)$ many times and produce many crossings in a longer
``rectangle" (it is obvious that if
$\Theta\subset\overline{\mathbb{D}_{2R}^{\eta}}$ this would not
happen). This observation and the above inequality lead to
\begin{align*}
&\hat{P}[|\theta(\gamma[0,\tau_{(1/2)^{j+1}R}])|\geq C_8j |\mathcal
{B}_j^c\mathcal
{B}_{j+1}]\\
&\leq \frac{C_{10}\exp(C_{11}j)}{P\left[\Gamma_1
\stackrel{\dot{\overline{\Theta}}}\leftrightarrow\partial
\mathbb{D}_{R/2}^{\eta}\right]}\left\{P\left[
\begin{aligned}
&\mbox{$\exists \lfloor C_8j/4\pi\rfloor-2$ disjoint yellow arms
crossing}\\
&A^{\eta}(R,2R)\mbox{ in } \dot{\overline{\Theta}},\Gamma_1
\stackrel{\dot{\overline{\Theta}}}\leftrightarrow\partial
\mathbb{D}_{(1/2)^{j+1}R}^{\eta}
\end{aligned}
\right]+\right.\\
&~~~~~~~~~~~~~~~~~~~~~~~~~~~~~~~~~~~~\left.P\left[
\begin{aligned}
&\mbox{$\exists\lfloor C_8j/4\pi\rfloor-2$
disjoint yellow crossings of}\\
&\mbox{$R^{\eta}((1/2)^{j+1}R,2R)$ in
$\dot{\overline{\Theta}},\Gamma_1\stackrel{\dot{\overline{\Theta}}}\leftrightarrow\partial
\mathbb{D}_{(1/2)^{j+1}R}^{\eta}$}
\end{aligned}
\right]\right\}\\
&\leq C_{10}\exp(C_{11}j)(\exp(-C_{12}j)+\exp(-C_{13}j))\mbox{ by BK
inequality, (\ref{e10}) and Lemma \ref{l20}}\\
&\leq C_{14}\exp(-C_{15}j).
\end{align*}
Then (\ref{e50}) follows immediately.  The proof of (\ref{e51}) is
similar to that of (\ref{e50}), the details are left to the reader.
\end{proof}

\subsection{Decorrelation of winding}\label{s44}

To simplify notation, we write $T_j:=T_{\epsilon^j,\epsilon^{j+1}}$
and $\tau_j:=\tau_{\epsilon^j}$ in the following.  The two lemmas
below say that $Var[\theta_{\eta}]$ is well-approximated by the sum
of the second moment of the winding numbers of the paths in annuli
on dyadic scales.

\begin{lemma}\label{l14}
There exists a constant $C>0$, such that for all
$10\eta<\epsilon<1/2$, under the conditional law $P[\cdot|\mathcal
{A}^{\eta}]$, we have
\begin{equation*}
\left|Var[\theta_{\eta}]-\sum_{j=0}^{\lfloor\log_{\epsilon}\eta\rfloor}E\left[\theta(\gamma[T_j,\tau_{j+1}])^2\right]\right|\leq
C[\log(1/\epsilon)]^{-\frac{1}{2}}\log(1/\eta).
\end{equation*}
\end{lemma}

\begin{proof}
Lemma \ref{l22} says that $|E\theta_{\eta}|\leq \pi$.  Therefore, in
order to prove Lemma \ref{l14}, it is enough to prove that there
exists a constant $C_1>0$, such that for all $10\eta<\epsilon<1/2$,
\begin{equation}\label{e33}
\left|E\left[\theta_{\eta}^2\right]-\sum_{j=0}^{\lfloor\log_{\epsilon}\eta\rfloor}E\left[\theta(\gamma[T_j,\tau_{j+1}])^2\right]\right|\leq
C_1[\log(1/\epsilon)]^{-\frac{1}{2}}\log(1/\eta).
\end{equation}
It is clear that
\begin{equation*}
\theta_{\eta}=\sum_{j=0}^{\lfloor\log_{\epsilon}\eta\rfloor}\theta(\gamma[\tau_j,\tau_{j+1}]).
\end{equation*}
So, to show (\ref{e33}), it suffices to prove that there are
$C_2,C_3>0$, such that for all $10\eta<\epsilon<1/2$,
\begin{equation}\label{e34}
\sum_{j=0}^{\lfloor\log_{\epsilon}\eta\rfloor}\left|E\left[\theta(\gamma[\tau_j,\tau_{j+1}])^2\right]-E\left[\theta(\gamma[T_j,\tau_{j+1}])^2\right]\right|\leq
C_2[\log(1/\epsilon)]^{-\frac{1}{2}}\log(1/\eta)
\end{equation}
and
\begin{equation}\label{e39}
\left|\sum_{0\leq
j<k\leq\lfloor\log_{\epsilon}\eta\rfloor}E[\theta(\gamma[\tau_j,\tau_{j+1}])\theta(\gamma[\tau_k,\tau_{k+1}])]\right|\leq
C_3\log_{\epsilon}\eta.
\end{equation}
Let us first show (\ref{e34}).   By (\ref{e74}), (\ref{e58}) and
Cauchy-Schwarz inequality, there exist $C_4,C_5>0$, such that
\begin{align*}
&E\left|\theta(\gamma[\tau_j,\tau_{j+1}])^2-\theta(\gamma[T_j,\tau_{j+1}])^2\right|\\
&~~~~=E\left|2\theta(\gamma[\tau_j,T_j])\theta(\gamma[T_j,\tau_{j+1}])+\theta(\gamma[\tau_j,T_j])^2\right|\\
&~~~~\leq
2\left\{E\left[\theta(\gamma[\tau_j,T_j])^2\right]\right\}^{\frac{1}{2}}\left\{E\left[\theta(\gamma[T_j,\tau_{j+1}])^2\right]\right\}^{\frac{1}{2}}+E\left[\theta(\gamma[\tau_j,T_j])^2\right]\\
&~~~~\leq C_4[\log(1/\epsilon)]^{\frac{1}{2}}+C_5.
\end{align*}
Then we get (\ref{e34}).

Now let us show (\ref{e39}).  For this, it is enough to show that
there are $C_6,C_7>0$, such that for any $0\leq
j\leq\lfloor\log_{\epsilon}\eta\rfloor$,
\begin{equation}\label{e40}
\left|E[\theta(\gamma[\tau_j,\tau_{j+1}])\theta(\gamma[\tau_{j+1},\tau_{\eta}])]\right|\leq
C_6,
\end{equation}
\begin{equation}\label{e41}
\left|E[\theta(\gamma[0,\tau_j])\theta(\gamma[\tau_j,\tau_{j+1}])]\right|\leq
C_7.
\end{equation}
We first show (\ref{e40}).  Note that $\gamma[0,\tau_j]$ and the
b-path $\gamma'[0,\tau_j]$ from $(-1)_{\eta}$ to
$\partial\mathbb{D}_{\epsilon^j}^{\eta}$ induce faces $\Theta_j$
around $\partial\mathbb{D}_{\epsilon^j}^{\eta}$.  Denote by
$E_{\Theta_j}$ the expectation with respect to $P[\cdot|\mathcal
{A}_{\Theta_j}^{\eta}(\eta,\epsilon^j)]$. By Lemma \ref{l22} and
Lemma \ref{l18}, choosing $C_6,C_8$ appropriately, we have
\begin{align*}
&\left|E[\theta(\gamma[\tau_j,\tau_{j+1}])\theta(\gamma[\tau_{j+1},\tau_{\eta}])]\right|\\
&~~~~\leq\sum_{\Theta_{j+1}}P[\Theta_{j+1}|\mathcal {A}^{\eta}]\left|E_{\Theta_{j+1}}[\theta(\gamma[\tau_j,\tau_{j+1}])]\right|\left|E_{\Theta_{j+1}}[\theta(\gamma[\tau_{j+1},\tau_{\eta}])]\right|\\
&~~~~\leq C_8\left|E[\theta(\gamma[\tau_j,\tau_{j+1}])]\right|\leq
C_6.
\end{align*}
Now let us show (\ref{e41}),  which proof is similar to that of
(\ref{e40}).  By Lemma \ref{l22} and Lemma \ref{l18} again, we have
\begin{align*}
\left|E[\theta(\gamma[0,\tau_j])\theta(\gamma[\tau_j,\tau_{j+1}])]\right|&\leq\sum_{\Theta_j}P[\Theta_j|\mathcal {A}^{\eta}]\left|E_{\Theta_j}[\theta(\gamma[0,\tau_j])]\right|\left|E_{\Theta_j}[\theta(\gamma[\tau_j,\tau_{j+1}])]\right|\\
&\leq C_9\left|E[\theta(\gamma[0,\tau_j])]\right|\leq C_7.
\end{align*}
\end{proof}

Denote by $E_j$ the expectation with respect to the conditional law
$P[\cdot|\mathcal {A}^{\eta}(\epsilon^j)]$.  Conditioned on
$\mathcal {A}^{\eta}(\epsilon^j)$, denote by $\gamma_j$ the
percolation exploration path
$\gamma^{\eta}_{\mathbb{D}_{\epsilon^j},\epsilon^j,-\epsilon^j}$
stopped when it reaches
$\partial\mathbb{D}_{\epsilon^{j+1}}^{\eta}$.
\begin{lemma}\label{l17}
There exist $C>0$ and $0<\epsilon_0<1/2$, such that for all
$10\eta<\epsilon<\epsilon_0$, we have
\begin{equation*}
\left|\sum_{j=0}^{\lfloor\log_{\epsilon}\eta\rfloor}E\left[\theta(\gamma[T_j,\tau_{j+1}])^2\right]-\sum_{j=0}^{\lfloor\log_{\epsilon}\eta\rfloor}E_j\left[\theta(\gamma_j)^2\right]\right|\leq
C[\log(1/\epsilon)]^{-\frac{1}{7}}\log(1/\eta).
\end{equation*}
\end{lemma}

\begin{proof}
Let $\beta$ be the constant in Proposition \ref{p3}.  By Proposition
\ref{p3}, we can couple $P[\cdot|\mathcal {A}^{\eta}]$ and
$P[\cdot|\mathcal {A}^{\eta}(\epsilon^j)]$ such that with
probability at least $1-\epsilon^{\beta/3}$ there exist identical
good faces $\Theta\subset A^{\eta}(\epsilon^{j+1/3},\epsilon^j)$ for
both measures, and the configuration in $\overline{\Theta}$ is also
identical.  Let us denote by $\hat{P}$ the coupling law, by
$\hat{E}$ the expectation with respect to $\hat{P}$, and by
$\mathcal {B}$ the above event. We write
\begin{align*}
\hat{E}&\left|\theta(\gamma[T_j,\tau_{j+1}])^2-\theta(\gamma_j)^2\right|\\
&=\hat{E}\left[I_{\mathcal
{B}^c}\left|\theta(\gamma[T_j,\tau_{j+1}])^2-\theta(\gamma_j)^2\right|\right]+
\hat{E}\left[I_{\mathcal
{B}}\left|\theta(\gamma[T_j,\tau_{j+1}])^2-\theta(\gamma_j)^2\right|\right].
\end{align*}
Let us estimate the two terms in the r.h.s. of above equality
separately.  For the first term, with Cauchy-Schwarz inequality and
(\ref{e75}), we get
\begin{align*}
\hat{E}&\left[I_{\mathcal
{B}^c}\left|\theta(\gamma[T_j,\tau_{j+1}])^2-\theta(\gamma_j)^2\right|\right]\\
&\leq \left\{\hat{P}\left[\mathcal
{B}^c\right]\right\}^{\frac{1}{2}}\left\{\hat{E}\left[\left|\theta(\gamma[T_j,\tau_{j+1}])^2-\theta(\gamma_j)^2\right|^2\right]\right\}^{\frac{1}{2}}\leq
C_1\epsilon^{\frac{\beta}{6}}[\log(1/\epsilon)]^2.
\end{align*}
Now let us bound the second term.  For each $x>0$, define event
$$\mathcal {S}_x:=\left\{\mbox{$\exists\lfloor x/2\pi\rfloor-4$ disjoint blue crossings of }R^{\eta}(\epsilon^{j+1/3},\epsilon^j)\right\}.$$
There exist $C_2,C_3,C_4,C_5>0$ such that for all
$10\eta<\epsilon<1/2$ and all $x\geq
C_2[\log(1/\epsilon)]^{\frac{1}{2}-\frac{1}{7}}$,
\begin{align*}
&\hat{P}\left[\mathcal
{B},\left|\theta(\gamma[T_j,\tau_{j+1}])-\theta(\gamma_j)\right|\geq
x\right]\\
&~~~~~~~\leq P[\mathcal {S}_x|\mathcal {A}^{\eta}]+P[\mathcal
{S}_x|\mathcal
{A}^{\eta}(\epsilon^j)]\\
&~~~~~~~\leq \frac{C_3P[\mathcal {S}_x]}{P[\mathcal
{A}_2^{\eta}(\epsilon^{j+1},\epsilon^j)]}\mbox{ by
quasi-multiplicativity and (\ref{e72})}\\
&~~~~~~~\leq C_4\exp(-C_5x)\mbox{ by (\ref{e10}) and Lemma
\ref{l20}}.
\end{align*}
Combining (\ref{e74}), Cauchy-Schwarz inequality and above
inequality, we have
\begin{align*}
\hat{E}&\left[I_{\mathcal
{B}}\left|\theta(\gamma[T_j,\tau_{j+1}])^2-\theta(\gamma_j)^2\right|\right]\\
&\leq
\left\{\hat{E}\left[\left|\theta(\gamma[T_j,\tau_{j+1}])+\theta(\gamma_j)\right|^2\right]\right\}^{\frac{1}{2}}
\left\{\hat{E}\left[I_{\mathcal
{B}}\left|\theta(\gamma[T_j,\tau_{j+1}])-\theta(\gamma_j)\right|^2\right]\right\}^{\frac{1}{2}}\\
&\leq C_5[\log(1/\epsilon)]^{1-\frac{1}{7}}.
\end{align*}
This, together with the upper bound for the first term completes the
proof immediately.
\end{proof}

\subsection{Proofs of Theorem \ref{t11} and Corollary
\ref{c4}}\label{s45}

We now conclude the proof of our main result concerning winding
numbers.

\begin{proof}[Proof of Theorem \ref{t11}]
By Lemma \ref{l14} and Lemma \ref{l17}, to establish (\ref{e29}), it
is enough to show that for each $0<\delta<1$, there exists
$0<\epsilon_0(\delta)<1$ such that for each given
$0<\epsilon<\epsilon_0$, there exists $\eta_0(\epsilon)>0$, such
that for all $\eta<\eta_0$,
\begin{equation}\label{e35}
\left|\frac{3}{2}\lfloor\log_{\epsilon}\eta\rfloor\log(1/\epsilon)-\sum_{j=0}^{\lfloor\log_{\epsilon}\eta\rfloor}E_j\left[\theta(\gamma_j)^2\right]\right|\leq\delta\lfloor\log_{\epsilon}\eta\rfloor\log(1/\epsilon).
\end{equation}
By (\ref{e74}), there exists a constant $C>0$, such that for all
$\eta<\epsilon<1/2$ and $0\leq
j\leq\lfloor\log_{\epsilon}\eta\rfloor$,
\begin{equation}\label{e38}
E_j\left[\theta(\gamma_j)^2\right]\leq C\log(1/\epsilon).
\end{equation}
Combining (\ref{e38}) and Lemma \ref{l16}, we have that for any
fixed $0<\epsilon<1/2$, for any $j$ such that
$\lfloor\log_{\epsilon}\eta\rfloor-j\rightarrow+\infty$ as
$\eta\rightarrow 0$,
\begin{equation}\label{e37}
E_j\left[\theta(\gamma_j)^2\right]\rightarrow
E^*\left[\theta(\gamma[0,\tau_{\epsilon}])^2\right]\mbox{ as
}\eta\rightarrow 0,
\end{equation}
where $\gamma$ is the two-sided radial SLE$_6$ path from $1$ to $-1$
through $0$ in $\overline{\mathbb{D}}$.  By the convergence of the
Ces\`{a}ro mean and (\ref{e37}), we have
\begin{equation*}
\lim_{\eta\rightarrow0}\frac{\sum_{j=0}^{\lfloor\log_{\epsilon}\eta\rfloor}E_j\left[\theta(\gamma_j)^2\right]}{\lfloor\log_{\epsilon}\eta\rfloor}=E^*\left[\theta(\gamma[0,\tau_{\epsilon}])^2\right].
\end{equation*}
Combining this and Lemma \ref{l13} gives (\ref{e35}).

Using the approach in the 2-arm case in \cite{24} with a little
modification, one can show that under $P[\cdot|\mathcal
{A}^{\eta}]$,
\begin{equation*}
\frac{\theta_{\eta}}{\sqrt{Var
\theta_{\eta}}}\rightarrow_{d}N(0,1)\mbox{ as }\eta\rightarrow 0.
\end{equation*}
Then (\ref{e30}) follows from this and (\ref{e29}).
\end{proof}

\begin{proof}[Proof of Corollary \ref{c4}]
Let $\tilde{\theta}_{\eta,1}$ and $\tilde{\theta}_{\eta,\nu}$ denote
$\tilde{\theta}_{\eta}$ under $P[\cdot|\mathcal {A}^{\eta}_2]$ and
$\nu_2^{\eta}$, respectively.  First we prove the corollary for
$\tilde{\theta}_{\eta,1}$.  By Lemma 3.4 in \cite {24} and Lemma
\ref{l22}, we know $|E'\tilde{\theta}_{\eta,1}|\leq 2\pi$ and
$|E\theta_{\eta}|\leq \pi$, where $E'$ is the expectation with
respect to $P[\cdot|\mathcal {A}^{\eta}_2]$.  Combining this,
Theorem 1.1 in \cite{24} and Theorem \ref{t11}, to show the
corollary for $\tilde{\theta}_{\eta,1}$, it is enough to show that
there exists a constant $C>0$, such that for all small $\eta$,
\begin{equation}\label{e76}
\left|E\left[\theta_{\eta}^2\right]-E'\left[\tilde{\theta}_{\eta,1}^2\right]\right|\leq
C[\log(1/\eta)]^{\frac{6}{7}}.
\end{equation}
The proof of (\ref{e76}) is analogous to that of Lemma \ref{l17}, we
just sketch it here:  By Proposition \ref{p3}, one can couple
$P[\cdot|\mathcal {A}^{\eta}]$ and $P[\cdot|\mathcal {A}^{\eta}_2]$
such that with probability at least $1-\eta^{\beta/3}$ there exist
identical good faces $\Theta\subset A^{\eta}(\eta^{1/3},1)$ for both
measures, and the configuration in $\overline{\Theta}$ is also
identical. Denote by $\hat{P}$ the coupling law, by $\hat{E}$ the
expectation with respect to $\hat{P}$, and by $\mathcal {B}$ the
above event. Then one can show that there exist $C_1,C_2>0$ such
that
\begin{equation*}
\hat{E}\left[I_{\mathcal
{B}^c}\left|\theta_{\eta}^2-\tilde{\theta}_{\eta,1}^2\right|\right]\leq
C_1\eta^{\frac{\beta}{6}}[\log(1/\eta)]^2,
\end{equation*}
\begin{equation*}
\hat{E}\left[I_{\mathcal
{B}}\left|\theta_{\eta}]^2-\tilde{\theta}_{\eta,1}^2\right|\right]\leq
C_2[\log(1/\eta)]^{\frac{6}{7}},
\end{equation*}
which imply (\ref{e76}) immediately.

Now let us show the corollary for $\tilde{\theta}_{\eta,\nu}$, which
proof is similar to that for $\tilde{\theta}_{\eta,1}$.  It is easy
to show that $|E_{\nu}\tilde{\theta}_{\eta,\nu}|\leq 2\pi$, where
$E_{\nu}$ is the expectation with respect to $\nu_2^{\eta}$.
Combining this, $|E\theta_{\eta}|\leq \pi$, Corollary 1.5 in
\cite{24} and Theorem \ref{t11}, to show the corollary for
$\tilde{\theta}_{\eta,\nu}$, it is enough to show that there exists
a constant $C_3>0$, such that for all small $\eta$,
\begin{equation}\label{e77}
\left|E\left[\theta_{\eta}^2\right]-E_{\nu}\left[\tilde{\theta}_{\eta,\nu}^2\right]\right|\leq
C_3[\log(1/\eta)]^{\frac{6}{7}}.
\end{equation}
For $n\geq 1$, let $\tilde{\theta}_{\eta,n}$ denote
$\tilde{\theta}_{\eta}$ under $P[\cdot|\mathcal
{A}^{\eta}_2(\eta,n)]$.  Similar to the proof of (\ref{e76}), one
can show that there is a $C_3>0$ such that for all $n\geq 1$ and all
small $\eta$,
\begin{equation*}
\left|E\left[\theta_{\eta}^2\right]-E_n\left[\tilde{\theta}_{\eta,n}^2\right]\right|\leq
C_3[\log(1/\eta)]^{\frac{6}{7}},
\end{equation*}
where $E_k$ is the expectation with respect to $P[\cdot|\mathcal
{A}^{\eta}_2(\eta,n)]$.  Then one obtains (\ref{e77}) from the above
inequality by taking $n\rightarrow\infty$.
\end{proof}

\section*{Acknowledgement}   The author thanks Geoffrey Grimmett for his invitation
to visit the Statistical Laboratory in Cambridge University,  where
part of this work was completed.  He also thanks the anonymous
referees for helpful comments.  The author was supported by the
National Natural Science Foundation of China (No. 11601505) and the
Key Laboratory of Random Complex Structures and Data Science, CAS
(No. 2008DP173182).

\end{document}